\documentclass[12pt]{amsart}
\usepackage{amsfonts}
\usepackage{fullpage,multirow, amsmath, amsthm,amsfonts,amssymb,stmaryrd, mathrsfs,hhline}
\usepackage{mathdots}
\usepackage{pgf,tikz,tikz-cd}
\usepackage{mathrsfs}
\usepackage{url}
\usepackage{mathtools}
\usepackage{float}
\usetikzlibrary{arrows}
\usetikzlibrary[patterns]
\usetikzlibrary{matrix,decorations.pathreplacing,calc,shapes, decorations.pathmorphing}
\usepackage{hyperref}

\usepackage[all]{xy}

\usepackage[all]{xy}

\newtheorem{theorem}[subsection]{Theorem}
\newtheorem{lemma}[subsection]{Lemma}
\newtheorem{corollary}[subsection]{Corollary}

\newtheorem{proposition}[subsection]{Proposition}
\theoremstyle{definition}

\newtheorem{definition}[subsection]{Definition}
\newtheorem{definition-proposition}[subsection]{Definition-Proposition}

\newtheorem{question}[subsection]{Question}
\newtheorem{remark}[subsection]{Remark}

\newtheorem{notation}[subsection]{Notation}

\numberwithin{equation}{subsection}

\makeatletter

\newcommand{\Rmnum}[1]{\expandafter\@slowromancap\romannumeral #1@}
\makeatother

\def\calO{\mathcal{O}}

\def\calV{\mathcal{V}}

\def\calX{\mathcal{X}}

\def\gotha{\mathfrak{a}}
\def\gothm{\mathfrak{m}}

\def\gothp{\mathfrak{p}}
\def\goths{\mathfrak{s}}

\def\AAA{\mathbb{A}}

\def\CC{\mathbb{C}}
\def\FF{\mathbb{F}}

\def\QQ{\mathbb{Q}}

\def\ZZ{\mathbb{Z}}

\def\scrS{\mathscr{S}}

\def\a{\alpha}
\def\b{\beta}
\def\c{\gamma}

\def\e{\varepsilon}

\def\vp{\varpi}

\def\t{\tau}
\def\w{\mathbf{w}}
\def\tp{\tilde\gothp}

\def\p{\gothp}

\def\Rvv{R^{\ver}(\mathbf{w},\tau)}

\DeclareMathOperator{\Gal}{Gal}

\DeclareMathOperator{\Hom}{Hom}

\DeclareMathOperator{\NP}{NP}

\DeclareMathOperator{\Spec}{Spec}
\DeclareMathOperator{\Spf}{Spf}

\DeclareMathOperator{\Sym}{Sym}

\DeclareMathOperator{\diag}{diag}

\DeclareMathOperator{\ur}{ur}

\DeclareMathOperator{\cris}{cris}
\DeclareMathOperator{\Berk}{Berk}
\DeclareMathOperator{\ps}{ps}
\DeclareMathOperator{\ver}{ver}
\DeclareMathOperator{\irr}{irr}

\DeclareMathOperator{\Min}{Min}

\newcommand{\quash}[1]{}  
\newcommand{\an}{\mathrm{an}}

\newcommand{\cycl}{\mathrm{cycl}}

\newcommand{\new}{\mathrm{new}}

\newcommand{\Iw}{\mathrm{Iw}}
\newcommand{\unr}{\mathrm{unr}}
\newcommand{\ord}{\mathrm{ord}}

\newcommand{\rig}{\mathrm{rig}}

\newcommand{\cl}{\mathrm{cl}}

\newcommand{\tri}{\mathrm{tri}}
\newcommand{\Ext}{\mathrm{Ext}}

\newcommand{\Rkis}{R_{\bar\rho}^{\square,\psi}(\w,\t)}

\DeclareMathOperator{\GL}{GL}

\DeclareMathOperator{\Ind}{Ind}
\DeclareMathOperator{\Proj}{Proj}
\DeclareMathOperator{\univ}{univ}

\DeclareMathOperator{\I}{I}

\DeclareMathOperator{\tr}{Tr}

\DeclareMathOperator{\crr}{cr}

\DeclareMathOperator{\co}{co}

\begin{document}
	\title{Irreducible components of crystabelline deformation rings:
		the wildly ramified case}
	\author{Jiawei An} 
	\address{Jiawei An, Morningside Center of Mathematics,
		Chinese Academy of Sciences, 55 Zhongguancun East Road, Haidian District, Beijing 100190,
		China.}
	\email{jiaweian@amss.ac.cn}
	
	\begin{abstract}
		Let $p\geq11$ be a prime, and let
		$\bar\rho:\Gal(\bar\QQ_p/\QQ_p)\rightarrow\GL_2(\bar\FF_p)$
		be a reducible and very generic representation.
		In this paper, we use ghost series \cite{Lghost2} to study
		the irreducible components of Kisin's potentially crystalline
		deformation rings of $\bar\rho$.
		For wildly ramified principal series inertial types and
		arbitrary regular Hodge--Tate weights, we distinguish
		the components by the crystabelline slopes on their generic
		fibers and determine the total number of irreducible components.
		Combining this with an explicit formula for the total
		Hilbert--Samuel multiplicity of the special fiber
		obtained from the Breuil--M\'ezard conjecture,
		we compute the  Hilbert--Samuel multiplicity of the special fiber of each
		component. Every such multiplicity is one when $\bar\rho$ is non-split, and is either one or two when $\bar\rho$ is split, according to whether the generic fiber of the corresponding component parametrizes
		reducible or irreducible Galois representations.
	\end{abstract}
	\maketitle
	
	\setcounter{tocdepth}{1}
	\tableofcontents
	\section{Introduction}
	Let $p\geq5$ be a prime.
	One of the major breakthroughs in the $p$-adic Langlands program in dimension two is the proof of the
	\emph{Breuil--M\'ezard conjecture}
	(Theorem~\ref{thm-BM-intro}) through work of
	Kisin, Pa\v{s}k\=unas, and others (see
	\cite[\S~1]{HT-BMconj} for further references).
	The conjecture expresses the Hilbert--Samuel multiplicities
	of the special fibers of Kisin's
	\emph{potentially semistable (or crystalline) deformation rings}
	in terms of the Jordan--H\"older multiplicities in the mod-$p$
	reductions of certain locally algebraic representations
	of $\GL_2(\ZZ_p)$.
	The resulting control of multiplicities plays
	an important role in Taylor--Wiles--Kisin patching arguments
	and their applications to modularity lifting theorems.
	
	In this paper, we study a finer problem of determining the
	Hilbert--Samuel multiplicity of the special fiber of
	\emph{each irreducible component} of a \emph{potentially crystalline
		deformation ring}. We prove that, in certain cases, each such
	multiplicity is either one or two. 
	
	\subsection{Breuil--M\'ezard multiplicities of deformation rings}
	Let $E/\QQ_p$ be a sufficiently large finite extension, with ring
	of integers $\calO$, uniformizer $\varpi$, and residue field $\FF$.
	We use the convention that the $p$-adic cyclotomic character
	$\chi_{\cycl}$ has Hodge--Tate weight $-1$, and write $\omega$
	for its reduction modulo $\varpi$.
	
	Let
	$\bar\rho:\Gal(\bar\QQ_p/\QQ_p)\rightarrow\GL_2(\FF)$
	be a continuous representation, and let $R_{\bar\rho}^{\square}$
	be its universal framed deformation ring.
	A \emph{$p$-adic Hodge type} $(\w,\t,\psi)$ consists of a pair of integers $\w=(k_1,k_2)$ with $k_1<k_2$, an \emph{inertial type} $\t$, and a continuous character
	$\psi:\Gal(\bar\QQ_p/\QQ_p)\rightarrow\calO^\times$ satisfying
	$
	\psi|_{\I_{\QQ_p}}
	=\chi_{\cycl}^{-k_1-k_2-1}\det\tau$; see \S~\ref{subsection-Kisin-rings}.
	Kisin constructed a reduced, $\calO$-flat quotient
	$R_{\bar\rho}^{\square,\psi,\operatorname{cr}}(\w,\t)$
	of $R_{\bar\rho}^{\square}$ parametrizing potentially crystalline
	lifts of $\bar\rho$ with Hodge--Tate weights $\w$, inertial
	type $\t$, and determinant $\psi\chi_{\cycl}$. We call this ring $R$ the \emph{potentially 
		crystalline deformation ring} of type $(\w,\t,\psi)$ and we assume that $R\neq0$.
	Moreover, Kisin (\cite{kisin-deform}) proved that $R$ is an equidimensional, complete Noetherian local
	$\calO$-flat algebra, and $R[1/\vp]$ is regular.
	The \emph{Breuil--M\'ezard multiplicity} of $R$ is defined to be the Hilbert--Samuel multiplicity  $e(R/\varpi)$ of the
	special fiber. 
	Let $\sigma(\w,\t)^{\operatorname{cr}}$ be the finite dimensional irreducible $E$-representation of $\GL_2(\ZZ_p)$ attached to $(\w,\t)$, and $W(\bar\rho)$ be the set of \emph{Serre weights} of $\bar\rho$; see \S~\ref{subsection-BM-conj} for details.
	\begin{theorem}[Breuil--M\'ezard conjecture]
		\label{thm-BM-intro}
		There are integers $\mu_\sigma(\bar\rho)\geq0$,
		depending only on $\bar\rho$ and $\sigma$, such that for any $p$-adic Hodge type $(\w,\t,\psi)$, the following holds
		\[
		e\bigl(R/\varpi\bigr)
		=
		\sum_{\sigma\in W(\bar\rho)}
		[\bar\sigma(\w,\t)^{\operatorname{cr},\operatorname{ss}}:\sigma]\,
		\mu_\sigma(\bar\rho),
		\]
		where $[\bar\sigma(\w,\t)^{\operatorname{cr},\operatorname{ss}}:\sigma]$ is the multiplicity of $\sigma$ occurring in the semisimplification of the mod-$\vp$ reduction of $\sigma(\w,\t)^{\operatorname{cr}}$.
	\end{theorem}
	We refer to \cite{EG1} for the explicit values of $\mu_{\sigma}(\bar\rho)$. Using this result, the \emph{total} Breuil--M\'ezard multiplicity $e(R/\vp)$ can be computed
	by determining the Jordan--H\"older multiplicities of
	$\bar\sigma(\w,\t)^{\operatorname{cr},\operatorname{ss}}$; see Proposition~\ref{prop-halo-mult} for an \emph{explicit} formula in certain cases.
	The main objective of this paper is to refine this calculation by studying
	the following question.
	\begin{question}
		\label{question-1}
		For a minimal prime $\p\in\Min(R)$, what is the
		Breuil--M\'ezard multiplicity
		$
		e\bigl(R/(\p,\varpi)\bigr)
		$
		of the corresponding irreducible component?
	\end{question}
	Note that the regularity of $R[1/\varpi]$ yields
	the following additivity formula for Breuil--M\'ezard
	multiplicities (Lemma~\ref{lem-inequality-of-mult}):
	\begin{equation}
		\label{eq-00}
		e(R/\varpi)
		=
		\sum_{\p\in\Min(R)}e\bigl(R/(\p,\varpi)\bigr).
	\end{equation}
	Thus Question~\ref{question-1} asks how the total Breuil--M\'ezard multiplicity $e(R/\vp)$ 
	is distributed among the special fibers of irreducible components of $R$.
	
	\subsection{Statement of the main results}
	\label{subsection-introduction-main-results} Fix a 
	$\bar\rho:\Gal(\bar\QQ_p/\QQ_p)\rightarrow\GL_2(\FF)$ and a $p$-adic Hodge type $(\w,\t,\psi)$. We assume that $\tau=\theta_1\oplus\theta_2$ is a principal series inertial
	type, where $\theta_i:\ZZ_p^\times\longrightarrow E^\times$
	are finite order characters, viewed as characters of the inertial subgroup
	$\I_{\QQ_p}$ via local class field theory. Put $R=R^{\square,\psi,\operatorname{cr}}_{\bar\rho}(\w,\t)$.
	\begin{definition}
		We call $\bar\rho$
		\emph{reducible and very generic} if
		\[
		\bar\rho\simeq
		\begin{pmatrix}
			\unr(\bar\beta_1)\omega^{a+b+1}&*\\
			0&\unr(\bar\beta_2)\omega^b
		\end{pmatrix},
		\qquad
		2\leq a\leq p-5,\quad 0\leq b\leq p-2,
		\]
		where $\bar\beta_1,\bar\beta_2\in\FF^\times$.
		We call $\tau$ \emph{wildly ramified} if the least integer
		$m\geq1$ such that $\theta_1\theta_2^{-1}$ factors through
		$(\ZZ/p^m\ZZ)^\times$ satisfies $m\geq2$.
	\end{definition}

	\begin{theorem}[Theorem~\ref{thm-BM=1-non-split}, Theorem~\ref{thm-main-thm-split}]
		\label{thm-Main-results}
		Assume that $p\geq11$.
		Let $\bar\rho$ be reducible and very generic, and let
		$(\w,\t,\psi)$ be a $p$-adic Hodge type such that $\t$
		is a wildly ramified principal series type.
		Put
		$R=R_{\bar\rho}^{\square,\psi,\operatorname{cr}}(\w,\t)$.
		For every $\p\in\Min(R)$, write $R(\p):=R/\p$.
		Then
		\[
		e(R(\p)/\varpi)=
		\begin{cases}
			1,&\text{if $\bar\rho$ is non-split},\\
			1\text{ or }2,&\text{if $\bar\rho$ is split}.
		\end{cases}
		\]
		Moreover, if $\bar\rho$ is split, then
		$e(R(\p)/\varpi)=1$ (resp.\ $2$) if and only if some (equivalently, every)
		closed point of $\Spec R(\p)[1/\varpi]$ corresponds
		to a reducible (resp. an irreducible) Galois representation.
	\end{theorem}
	\begin{remark}
		\begin{enumerate}
			\item By Lemma~\ref{lem comm alg2}, whenever $e(R(\p)/\varpi)=1$,
			the ring $R(\p)$ is simply a formal power series ring over $\calO$. However, our method still cannot determine how these components intersect in $R$.
			\item 
			We also determine the number of irreducible components
			of $R$.
			If $\bar\rho$ is non-split, then
			$
			\#\Min(R)=e(R/\varpi)$ (Theorem~\ref{thm-BM=1-non-split}).
			If $\bar\rho$ is split, then
			$
			2\#\Min(R)-e(R/\varpi)\in\{0,1,2\}$; see
			Corollary~\ref{cor-split-compare-mult} and
			Theorem~\ref{thm-main-thm-split} for explicit formulas.
					\end{enumerate}
			\end{remark}
			
			\begin{remark}
	We emphasize that Theorem~\ref{thm-Main-results} imposes no further restriction
			on the Hodge--Tate weights $k_1<k_2$. The assumptions on $\bar\rho$ and the condition $p\geq11$ are inherited from \cite{Lghost2}, as we use the main results of loc. cit to study the crystabelline slopes. 
			According to \cite{Lghost2}, the behavior of crystabelline slopes
			for a principal series type $\tau$ is \emph{much easier} to analyze
			in the wildly ramified case than in the tamely ramified case;
			see Remarks~\ref{rmk-2} and~\ref{rmk-7} for further discussion.
			There is ongoing work dealing with general principal series types $\tau$ under the same assumptions on $p$ and $\bar\rho$.
	\end{remark}
	Theorem~\ref{thm-Main-results} also provides a \emph{multiplicity-one property
		for Breuil--M\'ezard cycles}.
	Recall the cycle formulation of the Breuil--M\'ezard
	conjecture in \cite{EG1}:
	\[
	Z(R/\varpi)
	=
	\sum_{\sigma\in W(\bar\rho)}
	[\bar\sigma(\w,\t)^{\operatorname{ss}}:\sigma]\,
	C(\sigma),
	\]
	where $Z(R/\varpi)$ denotes the cycle associated with $R/\vp$ and $C(\sigma)$ is the Breuil--M\'ezard cycle
	associated with the Serre weight $\sigma$.
	In our setting, $W(\bar\rho)$ consists of one weight $\sigma$ when
	$\bar\rho$ is non-split and two companion weights $\sigma$ and $\sigma^{\co}$ when
	$\bar\rho$ is split; see
	\S~\ref{subsection-explicit-formula-of-BM-mult}.
	Theorem~\ref{thm-Main-results}, together with the comparison
	of cycles in Remarks~\ref{rmk-5} and~\ref{rmk-6}, gives
	the following refinement.
	
	\begin{corollary}
		\label{cor-BM-cycles-intro}
		Under the hypotheses of Theorem~\ref{thm-Main-results},
		for every $\p\in\Min(R)$ we have
		\[
		Z\bigl(R/(\p,\varpi)\bigr)
		=
		n_\sigma(\p)\,C(\sigma)
		+	n_{\sigma^{\co}}(\p)\,C(\sigma^{\co}),\quad n_\sigma(\p), n_{\sigma^{\co}}(\p)\in\{0,1\}.
		\]
		Thus each Breuil--M\'ezard cycle occurs in the special
		fiber of each irreducible component of $R$ at most once.
	\end{corollary}
	
	\begin{remark}
		\label{rmk-0}
		The corresponding results also hold for unframed
		deformation rings in the non-split case and for versal
		deformation rings in the split case.
		Indeed, the framed deformation rings are formally smooth
		over the corresponding unframed or versal rings.
		This passage induces a bijection on irreducible
		components and preserves their Breuil--M\'ezard
		multiplicities; see
		\S~\ref{subsection-versal-deformation-ring}
		and Lemma~\ref{lem-1}.
	\end{remark}
	\subsection{Overview of the proof}
	\label{subsection-introduction-proof}
	
	Keep the notation of \S~\ref{subsection-introduction-main-results}.
	Write $\calX=\Spf(R)^{\rig}$, and let $\calV$ be the universal
	Galois representation over $\calX$.
	In \S~\ref{section-calculating-crystabelline-slopes}, we first
	construct a rigid analytic map
	\[
	\lambda:\calX\rightarrow\AAA^{1,\an},
	\]
	which interpolates the (crystalline) Frobenius eigenvalue on the
	$\theta_1$-isotypic line of the potentially crystalline
	Dieudonn\'e module $\mathbb{D}_{\operatorname{pcrys}}(\calV)$;
	see Proposition~\ref{prop-slope-map}.
	For a closed point $x\in\calX$, the associated
	\emph{crystabelline slope} is defined by
	$\goths(x)=k_2-v_p(\lambda(x))$.
	Then the main results of \cite[\S~7]{Lghost2} allow us to explicitly compute the slopes
	occurring in $\calX$ using \emph{local ghost series}; see
	Proposition~\ref{prop-crystabelline-slopes}. As a consequence, we prove the following result.
	
	\begin{proposition}[{Proposition~\ref{prop-BM-mult=halo-mult}
			and Corollary~\ref{cor-split-compare-mult}}]
		\label{prop-mult=mult-intro}
		The number of distinct slopes occurring in $\calX$ is finite.
		Let $e'(\calX)$ denote this number.
		Then $e'(\calX)=e(R/\vp)$ if $\bar\rho$ is non-split, and
		$2e'(\calX)-e(R/\varpi)\in\{0,1,2\}$ if $\bar\rho$ is split
		(see Corollary~\ref{cor-split-compare-mult} for explicit formulas).
	\end{proposition}
	
	Here we use Theorem~\ref{thm-BM-intro} to compute $e(R/\vp)$
	explicitly.
	The finiteness in Proposition~\ref{prop-mult=mult-intro} ensures that the slope is constant
	on each connected component of $\calX$, and thus $e'(\calX)$ provides a lower bound for the number of
	connected components of $\calX$; see
	Corollary~\ref{cor-naive-lower-bound}.
	By the regularity of $R[1/\vp]$, $\#\Min(R)$ is equal to the number of connected components of $\calX$; see \S~\ref{subsection-components-and-BM-mult}.
	Together with \eqref{eq-00}, we have
	\begin{equation}
		\label{eq-01}
		e(R/\varpi)
		=\sum_{\p\in\Min(R)}e(R(\p)/\varpi)
		\geq\#\Min(R)\geq e'(\calX).
	\end{equation}
	Theorem~\ref{thm-Main-results} for non-split $\bar\rho$ now
	follows, since $e(R/\vp)=e'(\calX)$ in this case by
	Proposition~\ref{prop-mult=mult-intro}. In particular, for non-split $\bar\rho$, we have 
	\begin{equation}
		\label{eq-04}
		e(R/\vp)=e'(\calX)=\#(\Min(R)).
	\end{equation}
	
	\begin{remark}
		A feature of our approach is that we distinguish the
		irreducible components of $R$ by the crystabelline slopes
		on their generic fibers.
		More precisely, for each $\p\in\Min(R)$, there is a
		well-defined slope $\goths(\p)$ such that all closed
		points of $\Spf(R/\p)^{\rig}$ have crystabelline slope
		$\goths(\p)$; see Remark~\ref{rmk-1} for further discussion.
	\end{remark}
	
	For split $\bar\rho$, we combine the non-split case of
	Theorem~\ref{thm-Main-results} with results of
	\cite{paskunas-2adic,HT-BMconj} relating the Kisin deformation
	rings for split and non-split residual representations,
	and again use a counting argument to complete the proof. We will prove the versal version of Theorem~\ref{thm-Main-results},
	from which the original statement follows by Remark~\ref{rmk-0}.
	For simplicity, we illustrate our approach here under the
	assumption that the reducible locus of $\calX$ is empty.

	Write
	$\bar\chi_1:=\unr(\bar\b_1)\omega^{a+b+1}$ and
	$\bar\chi_2:=\unr(\bar\b_2)\omega^b$.
	Put $\bar\rho_0=\bar\chi_1\oplus\bar\chi_2$, and let
	$\bar\rho_1$ and $\bar\rho_2$ be the two non-split extensions
	of these characters in opposite directions.
	For $i=0,1,2$, let $e_i$ denote the Breuil--M\'ezard
	multiplicity of the corresponding deformation ring. By a direct computation using Theorem~\ref{thm-BM-intro},
	we have (see Lemmas~\ref{lem-type-Serre-wt}(3)
	and~\ref{lem-multiplicity-difference})
	\begin{equation}
		\label{eq-02}
		e_1=e_2,\quad\textrm{and}\quad e_0=e_1+e_2.
	\end{equation}
	Let $A_0:=R^{\ver}(\w,\t)$ and $A_i:=R_i(\w,\t)$ for $i=1,2$
	be the versal Kisin deformation rings defined in
	\S~\ref{subsection-versal-deformation-ring}.
	Then \cite[\S~6]{paskunas-2adic} provides, for $i=1,2$, the following homomorphism
	\begin{equation}
		\label{eq-03}
		A_i
		\longrightarrow
		\frac{A_i[[x,y]]}{(xy-c)}
		\simeq A_0,
	\end{equation}
	where $c$ is the generator of the reducibility ideal of the pseudo-deformation ring $R^{\ps}$ in \S~\ref{subsection-versal-deformation-ring}.
	Using this description, we prove the following result.
	\begin{proposition}[Proposition~\ref{prop-key-irreducible},
		Corollary~\ref{cor-component-injection-1}]
		\label{prop-split-intro}
		For $i=1,2$, there is an injection
		$
		\Min(A_0)\hookrightarrow\Min(A_i)$
		sending $\p$ to its preimage
		$\tp$ under the map in \eqref{eq-03}.
		Moreover,
		$
		e\bigl(A_0/(\p,\vp)\bigr)\leq2$
		for every $\p\in\Min(A_0)$.
	\end{proposition}
	For $i=1,2$, combining Proposition~\ref{prop-split-intro} with \eqref{eq-01}, \eqref{eq-04},
	and \eqref{eq-02}, we obtain
	\[
	\begin{aligned}
		e_0
		&\overset{\eqref{eq-01}}{=}\sum_{\p\in\Min(A_0)}
		e\bigl(A_0/(\p,\vp)\bigr)\overset{\ref{prop-split-intro}}{\leq}2\#\Min(A_0)
		\overset{\ref{prop-split-intro}}{\leq}2\#\Min(A_i)
		\overset{\eqref{eq-04}}{=}2e_i\overset{\eqref{eq-02}}{=}e_0,
	\end{aligned}
	\]
	from which Theorem~\ref{thm-Main-results} follows in this case.
	\subsection*{Notation}
	Let $p\geq5$ be a prime number and we normalize the $p$-adic valuation by $v_p(p)=1$. 
	Let $E/\QQ_p$ be a sufficiently large finite extension,
	with ring of integers $\calO$, uniformizer $\varpi$, and
	residue field $\FF$.
	
	Write $G_{\QQ_p}:=\Gal(\overline{\QQ}_p/\QQ_p)$ and
	$\I_{\QQ_p}\subseteq G_{\QQ_p}$ for the inertia subgroup.
	We use the convention that the cyclotomic character
	$\chi_{\cycl}$ has Hodge--Tate weight $-1$.
	Its reduction modulo $\varpi$ is denoted by $\omega$,
	and the Teichm\"uller lift of $\omega$ by $\tilde\omega$.
	
	For a crystabelline representation $V$ of $G_{\QQ_p}$ with
	coefficients in a finite extension $E'/E$, write
	$
	\mathbb{D}_{\operatorname{pcrys}}(V)
	:=
	\bigl(B_{\mathrm{cris}}\otimes_{\QQ_p}V\bigr)^{
		G_{\QQ_p(\mu_{p^n})}}$ for a sufficiently large $n$. It is equipped with
	commuting $E'$-linear actions of the crystalline Frobenius
	$\phi$ and $\Gal(\QQ_p(\mu_{p^\infty})/\QQ_p)$.
	We also regard the latter as an action of $\I_{\QQ_p}$
	via the natural quotient map.
	
	Let $\Delta\simeq\FF_p^\times$ be the torsion subgroup of
	$\ZZ_p^\times$.
	For $\lambda\in\FF_p^\times$, write $[\lambda]\in\Delta$
	for its Teichm\"uller lift.
	We use $\exp(p)$ as a topological generator of
	$1+p\ZZ_p$.
	For an integer $x$, let $\{x\}\in\{0,\dots,p-2\}$
	denote its representative modulo $p-1$.
	
	Put $K=\GL_2(\ZZ_p)$ and
	$\Iw_m:=
	\left\{
	\begin{pmatrix}a&b\\c&d\end{pmatrix}\in K:
	c\in p^m\ZZ_p
	\right\}$ for $m\geq1$.
	Write
	$\Iw:=\Iw_1$.
	Let $\Iw(1)$ be the pro-$p$ Iwahori subgroup
	and write
	$
	\bar T:=\Iw/\Iw(1)=
	\left\{
	\begin{pmatrix}[\lambda]&0\\0&[\mu]\end{pmatrix}:
	\lambda,\mu\in\FF_p^\times
	\right\}$.
	
	For a Noetherian local ring $(A,\gothm)$, we write $e(A)$ for its
	Hilbert--Samuel multiplicity with respect to $\gothm$, and $\Min(A)$ for its set of minimal prime ideals.
	
	\subsection*{Acknowledgements}
	
	We are grateful to Liang Xiao for suggesting that slope
	calculations could be used to determine the Breuil--M\'ezard
	multiplicities of individual irreducible components of
	Kisin deformation rings.
	We thank Yongquan Hu for several valuable discussions
	on the structure of Kisin deformation rings.
	We also thank John Bergdall, Robert
	Pollack, and Rufei Ren for helpful conversations.
	
  \subsection*{AI Tools}
  We used ChatGPT-6 to help check proofs and references, identify typos, improve the English and presentation. We take full responsibility for the mathematical content and correctness of this paper.
	\section{Kisin deformation rings and the Breuil--M\'ezard conjecture}
	\label{section-preliminaries-of-Kisin-rings}
	In this section, we first recall the construction of (framed) potentially
	crystalline deformation rings and the Breuil--M\'ezard conjecture
	(Theorem~\ref{thm-BM}), which reduces the computation of their
	Breuil--M\'ezard multiplicities (i.e., the Hilbert--Samuel multiplicities
	of their special fibers) to calculations in the mod-$p$ representation
	theory of $\GL_2(\ZZ_p)$. In \S~\ref{subsection-components-and-BM-mult},
	we discuss the irreducible components of these rings and formulate our
	main problem: computing the Breuil--M\'ezard multiplicity of each
	irreducible component. 
	In \S~\ref{subsection-explicit-formula-of-BM-mult}, we specialize to the
	cases considered in this paper and apply Theorem~\ref{thm-BM} to obtain
	\emph{explicit} formulas for the Breuil--M\'ezard multiplicities of the
	corresponding deformation rings.
	
	\subsection{Kisin's potentially crystalline deformation rings}
	\label{subsection-Kisin-rings}
	A $p$-adic Hodge type $(\w,\t,\psi)$ consists of the following data:
	\begin{itemize}
		\item a pair of integers $\w=(k_1,k_2)$ with $k_1<k_2$;
		\item an inertial type $\t$, namely, a continuous representation
		$\tau:\I_{\QQ_p}\rightarrow\GL_2(E)$ with open kernel that extends to
		the Weil group $W_{\QQ_p}$;
		\item a continuous character
		$\psi:G_{\QQ_p}\rightarrow\calO^{\times}$ satisfying
		$\psi|_{\I_{\QQ_p}}=\chi_{\cycl}^{-k_1-k_2-1}\det\tau$.
	\end{itemize}
	Let $\bar\rho:G_{\QQ_p}\rightarrow\GL_2(\FF)$ be a continuous representation,
	and $R_{\bar\rho}^{\square}$ be its universal framed deformation ring.
	Denote by $R_{\bar\rho}^{\square,\psi}(\w,\t)$ the potentially
	\emph{crystalline} deformation ring of type $(\w,\t,\psi)$. It is the
	unique reduced, $\calO$-flat quotient of $R_{\bar\rho}^{\square}$ such that for every finite extension $E'/E$, a continuous local
	$\calO$-algebra homomorphism
	$x^*:R_{\bar\rho}^{\square}\rightarrow\calO_{E'}$ factors through
	$R_{\bar\rho}^{\square,\psi}(\w,\t)$ if and only if the corresponding
	representation $\rho_x:G_{\QQ_p}\rightarrow\GL_2(E')$ is potentially crystalline
	with Hodge--Tate weights $(k_1,k_2)$,
	$\operatorname{WD}(\rho_x)|_{\I_{\QQ_p}}\simeq\tau\otimes_E E'$, and
	$\det\rho_x=\psi\chi_{\cycl}$. Note that $\Rkis\neq0$ only if $\bar\psi\,\omega=\det\bar\rho$, and this implies the following relation
	\begin{equation}
		\label{eq-0}
		\omega^{-k_2-k_1}\cdot\det(\bar\tau)|_{\Delta}=(\det\bar\rho)|_{\Delta},
	\end{equation}
	where $\bar\tau$ is the mod-$\vp$ reduction of $\t$ and $\det(\bar\rho)$ is a character of $\QQ_p^{\times}$ via class field theory. We always assume that $\Rkis\neq0$ in this paper, and hence the above relation \eqref{eq-0} holds. 
	
	\begin{theorem}[{\cite[Theorem~3.3.8]{kisin-deform}}]
		\label{thm-Kisin}
		If $R_{\bar\rho}^{\square,\psi}(\w,\t)$ is nonzero, then it is a reduced,
		equi-dimensional, complete Noetherian local $\calO$-flat algebra, such that $R_{\bar\rho}^{\square,\psi}(\w,\t)[1/\varpi]$ is regular.
	\end{theorem}
	
	\begin{remark}
		We only work with potentially \emph{crystalline} deformation rings and omit
		the superscript indicating the crystalline condition from our
		notation. Note that the ring constructed in \cite[Corollary~2.7.7]{kisin-deform} is
		already reduced, by \cite[Theorem~3.3.8]{kisin-deform}.
	\end{remark}
	
	\subsection{The Breuil--M\'ezard conjecture}
	\label{subsection-BM-conj}
	The Breuil--M\'ezard conjecture expresses the Hilbert--Samuel multiplicity
	$e(R_{\bar\rho}^{\square,\psi}(\w,\t)/\varpi)$ of the special fiber in terms of
	representation-theoretic multiplicities.
	Let $\sigma(\t)$ be the finite-dimensional smooth irreducible
	$E$-representation of $K$ associated with the inertial type $\t$ by Henniart in the crystalline case (we omit the superscript $\operatorname{cr}$); see
	\cite[\S~1]{HT-BMconj}.
	Consider the finite-dimensional $K$-representation
	\[
	\sigma(\w,\t):=
	(\Sym^{k_2-k_1-1}E^2\otimes\operatorname{det}^{-k_2})\otimes_E\sigma(\t).
	\]
	We denote the reduction modulo $\varpi$ of a $K$-stable $\calO$-lattice in
	$\sigma(\w,\t)$ by $\bar\sigma(\w,\t)$, and its semisimplification by
	$\bar\sigma(\w,\t)^{\operatorname{ss}}$ (note that in \cite[\S~3]{EG1}, the cyclotomic
	character has Hodge--Tate weight $1$, and thus in their convention, our Galois representation
	has Hodge--Tate weights $-k_2<-k_1$. Hence the associated algebraic representation is
	$
	\Sym^{-k_1-(-k_2)-1}E^2\otimes\operatorname{det}^{-k_2}$).
	
	Recall that a \emph{Serre weight} is an irreducible smooth $\FF$-representation of $K$.
	For a residual Galois representation $\bar\rho$, let $W(\bar\rho)$ be the
	set of Serre weights attached to $\bar\rho$ as in \cite[\S~3.3]{EG1}.
	The Breuil--M\'ezard conjecture for $\GL_2(\QQ_p)$, established for $p\geq5$
	through work of Kisin, Pa\v{s}k\=unas, and Hu--Tan, relates the multiplicity
	of the special fiber to the multiplicities
	$[\bar\sigma(\w,\t)^{\operatorname{ss}}:\sigma]$ for $\sigma\in W(\bar\rho)$.
	See \cite[\S~1]{HT-BMconj} for a discussion and further references.
	
	\begin{theorem}
		\label{thm-BM}
		For $(\w,\t,\psi)$ and $\bar\rho$ as above, there are integers
		$\mu_{\sigma}(\bar\rho)\geq0$, independent of $(\w,\t,\psi)$, such that
		\[
		e(R_{\bar\rho}^{\square,\psi}(\w,\t)/\varpi)
		=\sum_{\sigma\in W(\bar\rho)}
		[\bar\sigma(\w,\t)^{\operatorname{ss}}:\sigma]
		\cdot\mu_{\sigma}(\bar\rho).
		\]
	\end{theorem}
	
	\subsection{Irreducible components and Breuil--M\'ezard multiplicities}
	\label{subsection-components-and-BM-mult}
	We simply write $R:=R_{\bar\rho}^{\square,\psi}(\w,\t)$ in this subsection.
	Let $\operatorname{Min}(R)=\{\p_1,\dots,\p_n\}$ be the set of minimal prime
	ideals of $R$. Since $R$ is $\calO$-torsion-free, localization induces a bijection
	\[
	\operatorname{Min}(R)\longrightarrow\operatorname{Min}(R[1/\varpi]),
	\qquad\p_i\longmapsto\p_iR[1/\varpi].
	\]
	By Theorem~\ref{thm-Kisin}, $R[1/\varpi]$ is regular. Consequently, the
	irreducible components of $\Spec R[1/\varpi]$ are pairwise disjoint and coincide
	with its connected components. In other words, we have
	\[
	R[1/\varpi]\cong\prod_{i=1}^n(R/\p_i)[1/\varpi].
	\]
	Let $\calX:=\Spf(R)^{\rig}$ be the rigid generic fiber of $\Spf R$. Since $R[1/\varpi]$ is regular and hence normal, it follows from
	\cite[Theorems~2.1.3 and~2.3.1]{conrad} that there is also a natural bijection
	\[
	\left\{
	\begin{gathered}
		\text{irreducible (equivalently, connected)}\\
		\text{components of }\Spec R[1/\varpi]
	\end{gathered}
	\right\}
	\xrightarrow{\sim}
	\left\{
	\begin{gathered}
		\text{irreducible (equivalently, connected)}\\
		\text{components of }\calX
	\end{gathered}
	\right\},
	\]
	which sends the component corresponding to $\p_i R[1/\varpi]$
	to $\Spf(R/\p_i)^{\rig}$.
	
	\begin{lemma}
		\label{lem-inequality-of-mult}
		With the notation above, we have
		\[
		e(R/\varpi)
		=\sum_{i=1}^n e\bigl(R/(\p_i+\varpi R)\bigr)
		\geq \#(\textrm{irreducible components of }\Rkis).
		\]
	\end{lemma}
	
	\begin{proof}
		The rings $R/\p_i$ are $\calO$-flat and have the same dimension as $R$.
		The equality then follows from the product decomposition of
		$R[1/\varpi]$ and Lemma~\ref{lem multiplicity-components}.
	\end{proof}
	
	\begin{definition}
		Let $A$ be
		$R_{\bar\rho}^{\square,\psi}(\w,\t)$ or $R_{\bar\rho}^{\square,\psi}(\w,\t)/\p$ for some minimal prime
		ideal $\p$. The \emph{Breuil--M\'ezard multiplicity} of $A$ is defined to be the
		Hilbert--Samuel multiplicity $e(A/\varpi)$ of the special fiber $A/\vp$.
	\end{definition}
	
	A primary aim of this paper is to determine how the total multiplicity
	$e(R/\varpi)$ is distributed among the irreducible components of $\Spec R$.
	For instance, if $e(R(\p)/\varpi)=1$ for some $\p\in\Min(R)$, then
	Lemma~\ref{lem comm alg2} implies that $R(\p)$ is a formal power series ring
	over $\calO$.
	
	\begin{remark}
		Even if every irreducible component is a formal power series ring over
		$\calO$, this does not determine the integral structure of $R$ itself.
		See \cite[Example~5.1.4]{BM02} for an illustration of this point.
	\end{remark}
	
	\subsection{Explicit formulas for $e(R_{\bar\rho}^{\square,\psi}(\w,\t)/\varpi)$}
	\label{subsection-explicit-formula-of-BM-mult}
	In this paper, we mainly focus on \emph{wildly ramified principal series types} and
	\emph{reducible, very generic} residual $\bar\rho$. In what follows, we first introduce these notions
	and then carry out the multiplicity calculation.
	
	\begin{definition}
		\label{def-types-and-rho}
		\begin{enumerate}
			\item A \emph{principal series type} is a direct sum
			$\t=\theta_1\oplus\theta_2$, where
			$\theta_i:\ZZ_p^{\times}\rightarrow E^{\times}$ are finite order characters, viewed as characters of $\I_{\QQ_p}$ via local class field
			theory. If $\theta_1\neq\theta_2$, its $p$-conductor is the least integer
			$m\geq1$ such that $\theta_1\theta_2^{-1}$ factors through
			$(\ZZ_p/p^m\ZZ_p)^{\times}$. If $\theta_1=\theta_2=\theta$, we set
			$m=0$. We call $\t$ \emph{wildly ramified} if $m\geq2$.
			
			\item A reducible representation $\bar\rho$ of $G_{\QQ_p}$ is called
			\emph{generic} if
			\[
			\bar\rho\simeq
			\begin{pmatrix}
				\unr(\bar\b_1)\omega^{a+b+1}&*\\
				0&\unr(\bar\b_2)\omega^b
			\end{pmatrix},
			\]
			where $a\in\{1,\dots,p-4\}$, $b\in\{0,\dots,p-2\}$, and
			$\bar\b_1,\bar\b_2\in\FF^{\times}$.
			It is called \emph{very generic} if $2\leq a\leq p-5$.
		\end{enumerate}
	\end{definition}
	We now compute the Breuil--M\'ezard multiplicity of
	$R=R^{\square,\psi}_{\bar\rho}(\w,\t)$ when $\t$ is a principal series type
	and $\bar\rho$ is reducible and generic. The following lemma reduces the
	calculation to representation-theoretic multiplicities via
	Theorem~\ref{thm-BM}.
	
	\begin{lemma}
		\label{lem-type-Serre-wt}
		Let $\t$ be a principal series type, and let $\bar\rho$ be a reducible
		generic Galois representation as in Definition~\ref{def-types-and-rho}.
		Then the following statements hold.
		\begin{enumerate}
			\item For $m\geq1$, denote by $\theta_1\otimes\theta_2$ the character of $\Iw_m$, which sends $g=(g_{i,j})_{1\leq i,j\leq2}$ to $\theta_2(\det g)(\theta_1\theta_2^{-1})(g_{11})$.
			We have
			\[
			\sigma(\w,\t)=
			\begin{cases}
				(\Sym^{k_2-k_1-1}E^2\otimes\operatorname{det}^{-k_2})
				\otimes(\theta\circ\det),&\text{if }m=0,\\[4pt]
				(\Sym^{k_2-k_1-1}E^2\otimes\operatorname{det}^{-k_2})
				\otimes\Ind_{\Iw_m}^K(\theta_1\otimes\theta_2),
				&\text{if }m\geq1.
			\end{cases}
			\]
			
			\item Put
			$
			\sigma:=\Sym^a\FF^2\otimes\operatorname{det}^b$, and
			$\sigma^{\mathrm{co}}:=\Sym^{p-3-a}\FF^2\otimes\operatorname{det}^{a+b+1}$.
			Then we have
			\[
			W(\bar\rho)=
			\begin{cases}
				\{\sigma\},&\text{if }\bar\rho\text{ is non-split},\\[4pt]
				\{\sigma,\sigma^{\mathrm{co}}\},&\text{if }\bar\rho\text{ is split}.
			\end{cases}
			\]
			
			\item The Breuil--M\'ezard multiplicity is
			\[
			e(R_{\bar\rho}^{\square,\psi}(\w,\t)/\varpi)=
			\begin{cases}
				[\bar\sigma(\w,\t)^{\operatorname{ss}}:\sigma],
				&\text{if }\bar\rho\text{ is non-split},\\[4pt]
				[\bar\sigma(\w,\t)^{\operatorname{ss}}:\sigma]
				+[\bar\sigma(\w,\t)^{\operatorname{ss}}:
				\sigma^{\mathrm{co}}],
				&\text{if }\bar\rho\text{ is split}.
			\end{cases}
			\]
		\end{enumerate}
	\end{lemma}
	
	\begin{proof}
		By \cite[\S~4.2]{hu-paskunas}, $\sigma(\tau)$ is given by
		\[
		\sigma(\t)=
		\begin{cases}
			\theta\circ\det,&\text{if }m=0,\\[4pt]
			\Ind_{\Iw_m}^K(\theta_1\otimes\theta_2),&\text{if }m\geq1.
		\end{cases}
		\]
		Then (1) follows from \cite[\S~3.1]{EG1}, and (2)
		follows from \cite[\S~3.3]{EG1}.
		For each $\sigma'\in W(\bar\rho)$ in (2), $\mu_{\sigma'}(\bar\rho)=1$ by
		\cite[Proposition~1.1.3(2)]{EG1}, and hence (3) follows from Theorem~\ref{thm-BM}.
	\end{proof}
	
	Let $\bar\theta_i$ denote the reduction of $\theta_i$ modulo $\varpi$, and
	write $\bar\theta_i=\omega^{n_i}$ for $i=1,2$. Here $\bar\theta_i$ is trivial on $1+p\ZZ_p$ and we view it as a character of $\Delta$. Moreover, we also view $\bar\theta_1\otimes\bar\theta_2$ as a character of $\Iw/\Iw(1)\simeq\bar T$.
	Since we always assume that $\Rkis\neq0$,
	the determinant compatibility \eqref{eq-0} implies that the central character of
	$\bar\sigma(\w,\t)$ agrees with that of $\sigma$, namely
	\begin{equation}
		\label{eq3}
		k_2-k_1-1-2k_2+n_1+n_2\equiv a+2b\pmod{p-1}.
	\end{equation}
	The following proposition (together with Lemma~\ref{lem-type-Serre-wt}(3)) provides
	explicit formulas for the Breuil--M\'ezard multiplicity of
	$R_{\bar\rho}^{\square,\psi}(\w,\t)$.
	
	\begin{proposition}
		\label{prop-halo-mult}
		Keep the notation of Definition~\ref{def-types-and-rho} and
		Lemma~\ref{lem-type-Serre-wt}. Write
		$k=k_2-k_1+1$ and $\bar\theta_i|_{\Delta}=\omega^{n_i}$ for $i=1,2$.
		Consider the following integer
		\[
		s=s(\w,\t,\sigma):=\{k_2-(n_1-b)\}\in\{0,\dots,p-2\}.
		\]
		Assume that $m\geq1$. Then we have
		\begin{align*}
			[\bar\sigma(\w,\t)^{\operatorname{ss}}:\sigma]
			& =2(k-1)\frac{p^{m-1}-1}{p-1}
			+\left\lfloor\frac{k-2-s}{p-1}\right\rfloor
			+\left\lfloor\frac{k-2-\{s+a\}}{p-1}\right\rfloor+2,\\
			[\bar\sigma(\w,\t)^{\operatorname{ss}}:\sigma^{\mathrm{co}}]
			& =2(k-1)\frac{p^{m-1}-1}{p-1}
			+\left\lfloor\frac{k-2-\{s-1\}}{p-1}\right\rfloor
			+\left\lfloor\frac{k-2-\{s+a+1\}}{p-1}\right\rfloor+2.
		\end{align*}
	\end{proposition}
	
	\begin{proof}
		Twisting by $\tilde\omega^{-b}$ replaces $\bar\rho$ by
		$\bar\rho\otimes\omega^{-b}$ and $\tau$ by
		$\tau\otimes\tilde\omega^{-b}|_{I_{\QQ_p}}$.
		On the representation-theoretic side, this tensors both
		$\bar\sigma(\w,\tau)$ and the Serre weights by $\det^{-b}$,
		and hence preserves the multiplicities in question.
		Since this twist also preserves $s$, we may assume that $b=0$.
		The definition of $s$ and \eqref{eq3} then give
		\begin{equation}\label{eq4}
			k_2-n_1\equiv s\pmod{p-1},
			\qquad
			n_2-k_1-1\equiv s+a\pmod{p-1}.
		\end{equation}
		Using the standard $K$-invariant $\calO$-lattice inside $\sigma(\w,\t)$ and the projection formula, we have
		\[
		\bar\sigma(\w,\tau)\simeq\Ind_{\Iw}^{K}\pi,\quad\textrm{where}\quad\pi=\Ind_{\Iw_m}^{\Iw}\left((\Sym^{k-2}\FF^2\otimes\operatorname{det}^{-k_2})|_{\Iw_m}\otimes(\bar\theta_1\otimes\bar\theta_2)|_{\Iw_m}\right).
		\]
		Since both $\Sym^{k-2}\FF^2\otimes\operatorname{det}^{-k_2}$ and $\bar\theta_1\otimes\bar\theta_2$ are $\Iw$-representations, the projection formula gives
		$
		\pi\simeq(\Sym^{k-2}\FF^2\otimes\operatorname{det}^{-k_2}\otimes(\bar\theta_1\otimes\bar\theta_2))|_{\Iw}\otimes\Ind_{\Iw_m}^{\Iw}\mathbf 1$.
		Let $\Proj_\sigma$ (resp. $\Proj_{\chi}$) be the projective envelope of $\sigma$ (resp.  a character $\chi$ of $\bar T$) in the category of profinite $\FF[\![K]\!]$-modules (resp. $\FF[\![\Iw]\!]$-modules). 
		Using the exactness of $\Hom_{\FF[\![K]\!]}(\Proj_\sigma,-)$ and Frobenius reciprocity, we have
		\begin{align*}
			\bigl[\bar{\sigma}(\mathbf{w},\tau)^{\mathrm{ss}}:\sigma\bigr]&=\dim_\FF\Hom_{\FF[\![K]\!]}(\Proj_\sigma,\bar\sigma(\w,\t))=\dim_\FF\Hom_{\FF[\![\Iw]\!]}(\Proj_\sigma|_{\Iw},\pi)
			\\&=\dim_{\FF}\Hom_{\FF[\![\Iw]\!]}\left(\Proj_{1\times\omega^a}\oplus\Proj_{\omega^a\times1},\pi\right)
			=\dim_{\FF} \left(\pi^{\bar T=1\times\omega^a}\oplus \pi^{\bar T=\omega^a\times1}\right),
		\end{align*}
		where we use $\Proj_\sigma|_{\Iw}\simeq\Proj_{1\times\omega^a}\oplus\Proj_{\omega^a\times1}$ from \cite[Lemma~6.4.1]{Pa04}. 	It remains to compute these eigenspace dimensions.
		
		Note that $\pi$ is a semisimple $\FF[\bar T]$-module as the order of $\bar T$ is coprime to $p$.
		By Lemma~\ref{lem-rep1}, we have an isomorphism of
		$\FF[\bar T]$-modules
		\[
		\pi\simeq
		(\Sym^{k_2-k_1-1}\FF^2\otimes\operatorname{det}^{-k_2})|_{\Iw}
		\otimes(\bar\theta_1\otimes\bar\theta_2)
		\otimes\bigoplus_{j=0}^{p-2}
		(\omega^j\otimes\omega^{p-1-j})^{\oplus m(j)}.
		\]
		We realize $\Sym^{k-2}\FF^2$ as the space of homogeneous
		polynomials of degree $k-2$ in $\FF[X,Y]$, equipped with
		the standard $K$-action.
		For $0\leq n\leq k-2$, the monomial $X^nY^{k-2-n}$,
		tensored with a vector in the $j$-th character summand
		is an eigenvector for $\diag(\alpha,\delta)\in\bar T$
		with eigenvalue (using \eqref{eq4})
		\[
		\alpha^{n-k_2+n_1+j}
		\delta^{k-2-n-k_2+n_2-j}
		=
		\alpha^{n-s+j}
		\delta^{\{s+a\}-n-j}.
		\]
		Hence computing the dimension of the
		$\bar T=\omega^a\times1$ eigenspace of $\pi$
		amounts to considering the solutions of the congruence
		\begin{equation}\label{eq1}
			n-s+j\equiv a\pmod{p-1},
			\qquad
			0\leq n\leq k-2,\quad 0\leq j\leq p-2.
		\end{equation}
		The other condition,
		$\{s+a\}-n-j\equiv0\pmod{p-1}$,
		is equivalent to \eqref{eq1}, since the central characters
		of $\sigma$ and $\pi$ agree.
		For each fixed $0\leq n\leq k-2$, there is a unique solution
		$j(n)\in\{0,\dots,p-2\}$ to \eqref{eq1}.
		When $j=0$, the number of solutions $0\leq n\leq k-2$
		to \eqref{eq1} is
		$
		\left\lfloor
		(k-2-\{s+a\})/(p-1)
		\right\rfloor+1$.
		Hence
		\begin{align*}
			\dim_\FF\pi^{\bar T=\omega^a\times1}
			&=
			\#\{0\leq n\leq k-2:j(n)\neq0\}\cdot m(j)
			+\#\{0\leq n\leq k-2:j(n)=0\}\cdot m(0)\\
			&=
			(k-1)m(j\neq0)
			+\#\{0\leq n\leq k-2:j(n)=0\}\cdot
			\bigl(m(0)-m(j\neq0)\bigr)\\
			&=
			(k-1)\frac{p^{m-1}-1}{p-1}
			+\left\lfloor
			\frac{k-2-\{s+a\}}{p-1}
			\right\rfloor+1,
		\end{align*}
		where we use $m(0)=m(j\neq0)+1=C(m)+1$
		from Lemma~\ref{lem-rep1}.
		The dimensions of
		$\pi^{\bar T=1\times\omega^a}$,
		$\pi^{\bar T=\omega^{-1}\times\omega^{a+1}}$,
		and $\pi^{\bar T=\omega^{a+1}\times\omega^{-1}}$
		(the last two of which are used to compute
		$\bigl[\bar\sigma(\w,\tau)^{\mathrm{ss}}:\sigma^{\co}\bigr]$)
		are obtained by similar arguments.
		The proposition follows.
	\end{proof}
	\begin{remark}
		Retain the notation of Proposition~\ref{prop-halo-mult}, and assume
		that $b=0$. When $m\leq 1$, we relate the multiplicity
		$\bigl[\bar{\sigma}(\mathbf{w},\tau)^{\mathrm{ss}}:\sigma\bigr]$
		to the dimension formulas in \cite[\S~4]{Lghost}, which play a key
		role in the definition of the \emph{local ghost series} therein.
		If $m=1$, put
		\[
		d_k^{\Iw}
		:=
		\dim_{\FF}\Hom_{\FF[\![\Iw]\!]}
		\bigl(\Proj_\sigma, \Sym^{k_2-k_1-1}\FF^2\otimes\operatorname{det}^{-k_2}\otimes\bar\theta_1\otimes\bar\theta_2)
		=
		\bigl[\bar{\sigma}(\mathbf{w},\tau)^{\mathrm{ss}}:\sigma\bigr].
		\]
		This is the quantity denoted by
		$d_k^{\Iw}(\e\cdot(1\times\omega^{2-k}))$
		in \cite[Proposition~4.1]{Lghost}, where
		$\e=\omega^{-s}\times\omega^{s+a}$. In particular, $d_k^{\Iw}=\frac{2k}{p-1}+O(1)$.
		If $m=0$, so that $\theta_1=\theta_2=\theta$, put
		\[
		d_k^{\ur}
		:=
		\dim_{\FF}\Hom_{\FF[\![K]\!]}
		\bigl(\Proj_\sigma,\bar{\sigma}(\mathbf{w},\tau)\bigr)=	\bigl[\bar{\sigma}(\mathbf{w},\tau)^{\mathrm{ss}}:\sigma\bigr].
		\]
		This is the quantity denoted by $d_k^{\ur}(\bar\theta|_{\Delta}\cdot\omega^{-k_2})$
		in \cite[(4.6.1)]{Lghost}; an explicit formula is given
		in \cite[Proposition~4.7]{Lghost}. In particular, $d_k^{\ur}=\frac{2k}{p^2-1}+O(1)$.
	\end{remark}
	\begin{remark}
		Keep the notation of the previous remark.
		In the case $\theta_1=\theta_2$, we may also define $d_k^{\Iw}$ by the same
		Hom-space formula and set $$	d_k^{\new}=d_k^{\Iw}-2d_k^{\ur}=\frac{2k}{p+1}+O(1),$$ which is denoted by $d_k^{\new}(\bar\theta|_{\Delta}\cdot\omega^{-k_2})$ in \cite[Proposition~4.7]{Lghost}.
		The quantity $d_k^{\new}$ accounts for the contribution of
		\emph{potentially semistable but not potentially crystalline}
		representations to the
		Breuil--M\'ezard multiplicity. More precisely, for a non-split $\bar\rho$, let
		$
		R^{\crr}:=R_{\bar\rho}^{\square,\psi}(\mathbf{w},\tau)$
		be the potentially crystabelline deformation ring considered above,
		and let $R$ be the corresponding potentially semistable deformation
		ring with the same $p$-adic Hodge type.
		Write $\operatorname{St}$ for the Steinberg representation of
		$\mathrm{GL}_2(\FF_p)$ over $\FF$, inflated to $K$.
		By the potentially semistable version of the
		Breuil--M\'ezard conjecture, we have (note that $\Ind_{\Iw}^{K}\mathbf 1\simeq\mathbf 1\oplus\operatorname{St}$)
		\[
		\begin{aligned}
			e(R/\vp)
			&=
			\bigl[
			\bigl(\bar{\sigma}(\mathbf{w},\tau)
			\otimes_{\FF}\operatorname{St}\bigr)^{\mathrm{ss}}
			:\sigma
			\bigr]=
			\dim_{\FF}\Hom_{\FF[\![K]\!]}
			\bigl(\Proj_\sigma,
			\bar{\sigma}(\mathbf{w},\tau)
			\otimes_{\FF}\operatorname{St}\bigr)\\
			&=
			\dim_{\FF}\Hom_{\FF[\![K]\!]}
			\bigl(\Proj_\sigma,
			\bar{\sigma}(\mathbf{w},\tau)
			\otimes_{\FF}\Ind_{\Iw}^{K}\mathbf 1\bigr)
			-
			\dim_{\FF}\Hom_{\FF[\![K]\!]}
			\bigl(\Proj_\sigma,\bar{\sigma}(\mathbf{w},\tau)\bigr)\\
			&=d_k^{\Iw}-d_k^{\ur}=d_k^{\ur}+d_k^{\new}=e(R^{\crr}/\vp)+d_k^{\new}.
		\end{aligned}
		\]
		Let $R^{\mathrm{ds}}$ be the reduced quotient of $R$ defined in  \cite[\S~5]{gee-geraghty-quaternion}, which corresponds to the
		union of irreducible components on which the monodromy operator is
		generically nonzero. Then we have $d_k^{\new}=e(R^{\operatorname{ds}}/\vp)$.
	\end{remark}
	
	\begin{lemma}
		\label{lem-rep1}
		Suppose that $m\geq1$. Put $C(m)=\frac{p^{m-1}-1}{p-1}$. As an $\FF[\bar T]$-module, we have
		\begin{equation}
			\label{eq-lem1}
			V:=\Ind_{\Iw_m}^{\Iw}\mathbf1
			\simeq\bigoplus_{i=0}^{p-2}
			(\omega^i\otimes\omega^{p-1-i})^{\oplus m(i)},
		\end{equation}
		where $m(0)=C(m)+1$, and $m(i\neq0):=m(i)=C(m)$ for $1\leq i\leq p-2$.
	\end{lemma}
	
	\begin{proof}
		Since $\bar T$ is abelian and has order prime to $p$, the $\FF[\bar T]$-module $V$ admits a
		decomposition of the form \eqref{eq-lem1}.
		To determine the multiplicities $m(i)$, it suffices to restrict to
		$\Delta$, where we identify $\Delta$ with the subgroup $\operatorname{diag}(\Delta,1)$
		of $\bar T$.
		
		We claim that the action of $\Delta$ on the coset space
		$\Iw_m\backslash\Iw$ has one fixed point $\{0\}$ and
		$(p^{m-1}-1)/(p-1)$ free orbits (i.e. the stabilizer is trivial), from which we obtain
		\[
		V|_{\Delta}\cong
		\mathbf1\oplus
		\bigl(\FF[\FF_p^{\times}]\bigr)^{\oplus\frac{p^{m-1}-1}{p-1}},
		\qquad
		\FF[\FF_p^{\times}]\cong\bigoplus_{r=0}^{p-2}\omega^r.
		\]
		This gives the required values of $m(i)$.
		To prove the claim, use the coset decomposition
		\[
		\Iw=\bigsqcup_{t\in\ZZ_p/p^{m-1}\ZZ_p}\Iw_m u(t),
		\qquad
		u(t):=\begin{pmatrix}1&0\\pt&1\end{pmatrix}.
		\]
		Under this decomposition, right multiplication by
		$\operatorname{diag}([\lambda],1)$, for $\lambda\in\FF_p^{\times}$,
		sends $u(t)$ to $u([\lambda]t)$, where $[\lambda]$ denotes the Teichm\"uller
		lift of $\lambda$. Indeed,
		\[
		u(t)\begin{pmatrix}[\lambda]&0\\0&1\end{pmatrix}
		=\begin{pmatrix}[\lambda]&0\\pt[\lambda]&1\end{pmatrix}
		=\begin{pmatrix}[\lambda]&0\\0&1\end{pmatrix}u([\lambda]t).
		\]
		Clearly, the element $t=0$ is fixed. If $\lambda\neq1$, then $[\lambda]-1$ is
		a unit in $\ZZ_p$, so no nonzero element of
		$\ZZ_p/p^{m-1}\ZZ_p$ is fixed by such $\lambda$.
		This proves the claim and the lemma.
	\end{proof}
	
	\subsection{Setup}
	
	For the remainder of the paper, we fix data $(\bar\rho,\t,\w,\psi)$ satisfying
	the following conditions and assume that
	$R_{\bar\rho}^{\square,\psi}(\w,\t)\neq0$:
	\begin{itemize}
		\item $\t=\theta_1\oplus\theta_2$ is a wildly ramified principal series
		type with conductor $m\geq2$;
		\item $\w=(k_1,k_2)$ is a pair of Hodge--Tate weights with $k_1<k_2$;
		\item $\psi:G_{\QQ_p}\rightarrow\calO^{\times}$ is a continuous
		character satisfying
		$\psi|_{\I_{\QQ_p}}=\chi_{\cycl}^{-k_1-k_2-1}\det\tau$ and
		$\bar\psi\omega=\det\bar\rho$;
		\item $\bar\rho$ is reducible and very generic.
	\end{itemize}
	
	\section{Slopes of crystabelline representations}
	\label{section-calculating-crystabelline-slopes}
	In this section, we first construct a rigid analytic map
	$\calX\to\AAA^{1,\an}$, called the \emph{slope map}, where
	$\calX$ is the rigid generic fiber of
	$R_{\bar\rho}^{\square,\psi}(\w,\t)$;
	see \S~\ref{subsection-construction-of-slope-map}.
	We then compute the associated slopes using the ghost series
	of \cite{Lghost2} in \S~\ref{subsection-NP}, which relies on the local ghost theorem \cite[Theorem~2.6]{Lghost2}.
	The number of distinct slopes gives a \emph{naive lower bound}
	for the number of connected components of $\calX$.
	In \S~\ref{subsection-counting-slopes}, we compare this bound
	with the Breuil--M\'ezard multiplicity $e(\Rkis/\vp)$.
	For non-split $\bar\rho$, the two are equal
	(Proposition~\ref{prop-BM-mult=halo-mult}).
	For split $\bar\rho$, the lower bound is half the total
	multiplicity, up to a correction accounting for the boundary
	slopes (Corollary~\ref{cor-split-compare-mult}).
	
	\subsection{A slope map on Kisin deformation spaces}
	\label{subsection-construction-of-slope-map}
	Let $x$ be an $E'$-point of the rigid generic fiber
	\[
	\calX:=\calX_{\bar\rho}^{\square,\psi}(\w,\t)
	=\bigl(\Spf R_{\bar\rho}^{\square,\psi}(\w,\t)\bigr)^{\rig}.
	\]
	In what follows, we first attach a \emph{slope} $\goths(x)$ to the crystabelline
	$E'$-representation $\rho_x$. We then construct a rigid analytic map
	$\calX\to\AAA^{1,\an}$ from which these slopes can be recovered by
	taking $p$-adic valuations; see Proposition~\ref{prop-slope-map}.
	
	Consider
	$\mathbb{D}_{\operatorname{pcrys}}(\rho_x)=E'e_1\oplus E'e_2$,
	where $\I_{\QQ_p}$ acts on $E'e_i$ through $\theta_i$ and $\phi$ acts
	on this line with eigenvalue $\a_i$. Since $\theta_1\neq\theta_2$,
	the labeling of the inertia characters determines the labeling of the
	Frobenius eigenvalues.
	
	The representation $\rho_x$ is trianguline. It is either a direct sum
	of two crystabelline characters or is described by Colmez's parameter
	spaces $\scrS_+^{\ord}\cup\scrS_+^{\cris}$; see
	\cite[Proposition~3.5 and \S~4.5]{colmez-triangulline}.
	In the split case, $\bar\rho$ is necessarily split. In the non-split
	case, $\rho_x$ admits two triangulations by
	\cite[Proposition~4.13]{colmez-triangulline}. We use the following one,
	as in \cite[\S~7]{Lghost2}:
	\begin{equation}
		\label{eq:trangulation}
		\begin{aligned}
			0\longrightarrow\mathcal{R}_{E'}\bigl(x^{-k_1}(\theta_2,\a_2)\bigr)
			\longrightarrow\mathbb{D}_{\mathrm{rig}}^{\dagger}(\rho_x)
			\longrightarrow\mathcal{R}_{E'}\bigl(x^{-k_2}(\theta_1,\a_1)\bigr)
			\longrightarrow0.
		\end{aligned}
	\end{equation}
	Here $(\theta_i,\a_i)$ denotes the character of $\QQ_p^\times$ whose
	restriction to $\ZZ_p^\times$ is $\theta_i$ and whose value at $p$ is
	$\a_i$.
	
	\begin{definition}
		\label{def-slope-of-rep}
		The \emph{(crystabelline) slope} of $\rho_x$ is defined to be the $p$-adic valuation
		\[
		\goths(x):=k_2-v_p(\a_1).
		\]
		We simply say that a slope $\goths$ \emph{occurs in $\calX$} if
		$\goths(x)=\goths$ for some closed point $x\in\calX$.
		Write
		$
		\delta_1:=x^{-k_1}(\theta_2,\a_2)$, and
		$
		\delta_2:=x^{-k_2}(\theta_1,\a_1)$.
		Then
		we have $\goths(x)=v_p(\delta_1(p))$.
	\end{definition}
	
	The following standard result characterizes reducibility in terms of
	the crystabelline slope.
	
	\begin{lemma}
		\label{lem-crystabelline-slopes}
		Let $\rho_x$ be a crystabelline representation as above. Then $\goths(x)\in[0,k_2-k_1]$, and
		$\rho_x$ is reducible if and only if
		$\goths(x)\in\{0,k_2-k_1\}$.
	\end{lemma}
	
	\begin{proof}
		We include a proof for completeness. Weak admissibility of
		$\mathbb{D}_{\operatorname{pcrys}}(\rho_x)$ gives
		\begin{equation}
			\label{eq-proof1}
			\goths(x)\in[0,k_2-k_1]
			\qquad\text{and}\qquad
			v_p(\a_1)+v_p(\a_2)=k_1+k_2.
		\end{equation}
		Suppose first that $\rho_x$ is reducible. If it is split, then the
		line $D_1:=E'e_1$, with its induced filtration and $(\phi,\I_{\QQ_p})$-action,
		corresponds to a subrepresentation of $\rho_x$ with Hodge--Tate weight
		in $\{k_1,k_2\}$. Weak admissibility of $D_1$ therefore implies that
		$v_p(\a_1)\in\{k_1,k_2\}$, and hence
		$\goths(x)\in\{0,k_2-k_1\}$.
		If $\rho_x$ is non-split, the assertion follows from Colmez's
		classification in \cite[\S~4]{colmez-triangulline}. More precisely,
		for the trianguline parameter $s=(\delta_1,\delta_2,\infty)$ of $\rho_x$, with
		$\delta_1,\delta_2$ as in Definition~\ref{def-slope-of-rep}, we have
		\[
		\goths(x)=0\iff s\in\scrS_0^{\cris},
		\qquad
		\goths(x)=k_2-k_1\iff s\in\scrS_+^{\ord}.
		\]
		Conversely, if $\rho_x$ is irreducible, then
		its trianguline parameter $s\in\scrS_*^{\cris}$ by
		\cite[Proposition~4.8]{colmez-triangulline}, so
		$\goths(x)\in(0,k_2-k_1)$.
	\end{proof}
	
	Let $\calV$ be the universal Galois representation over $\calX$,
	obtained by restricting the universal
	$R_{\bar\rho}^{\square,\psi}$-representation $\rho^{\univ}$.
	We use $\calV$ to interpolate the eigenvalue $\a_1$ analytically over $\calX$.
	
	\begin{proposition}
		\label{prop-slope-map}
		There exists a rigid analytic map
		\[
		\lambda:=\lambda_{(\bar\rho,\psi,\tau,\w)}:
		\calX=\bigl(\Spf R_{\bar\rho}^{\square,\psi}(\w,\t)\bigr)^{\rig}
		\longrightarrow\AAA^{1,\an}
		\]
		such that, for every closed point $x\in\calX$, $k_2-v_p(\lambda(x))$ is equal to the slope $\goths(x)$ of $\rho_x$.
	\end{proposition}
	
	\begin{proof}
		It suffices to construct a global analytic function
		$\alpha\in\Gamma(\calX,\calO_{\calX})$ whose value at each closed
		point $x$ is the Frobenius eigenvalue $\a_1$ on the
		$\theta_1$-isotypic line of
		$\mathbb{D}_{\operatorname{pcrys}}(\rho_x)$.
		We follow the proof of
		\cite[Proposition~7.5.3]{rozensztajn-locus}, applied to the framed
		Kisin deformation space $\calX$.
		
		By \cite[Theorem~7.2.1]{rozensztajn-locus}, the family $\calV$ gives
		rise to a family $\mathbb{D}_{\operatorname{pcrys}}(\calV)$ over
		$\calX$ whose specialization at each closed point $x$ is
		$\mathbb{D}_{\operatorname{pcrys}}(\rho_x)$.
		Since $\theta_1\neq\theta_2$, applying
		\cite[Proposition~7.3.1]{rozensztajn-locus} to the inertia action
		gives a $\theta_1$-isotypic line bundle. This line bundle is preserved
		by $\phi$, since the inertia and Frobenius actions commute. Hence $\phi$ acts on this line bundle by multiplication by a
		global analytic function $\alpha$, which has the required properties.
	\end{proof}
	
	\begin{corollary}
		\label{cor-naive-lower-bound}
		Let $\calX^{\cl}$ denote the set of closed points of $\calX$.
		Suppose that the number of distinct slopes occurring in $\calX$, i.e.
		$\#(\{\goths(x):x\in\calX^{\cl}\})$, is finite.
		Then $\goths(x)$ is constant on each connected (equivalently,
		irreducible) component of $\calX$, and
		\begin{equation}
			\label{eq-proof2}
			e\bigl(R_{\bar\rho}^{\square,\psi}(\w,\t)/\vp\bigr)
			\geq
			\#\left\{\begin{array}{c}
				\text{connected components}\\
				\text{of }\calX
			\end{array}\right\}
			\geq\#(\{\goths(x):x\in\calX^{\cl}\}).
		\end{equation}
	\end{corollary}
	
	\begin{proof}
		The first inequality is Lemma~\ref{lem-inequality-of-mult}.
		For the second, consider the induced map of Berkovich spaces
		$
		\lambda^{\Berk}:\calX^{\Berk}\longrightarrow\mathbb A_E^{1,\Berk}$.
		Consider the continuous function
		\[
		f:\calX^{\mathrm{Berk}}\longrightarrow\mathbb R,
		\qquad
		x\longmapsto k_2+\log_p|\lambda^{\Berk}(x)|.
		\]
		Since the subspace of classical points (i.e. $\calX^\cl$) is dense in
		$\calX^{\mathrm{Berk}}$ and the image $\{f(x)=\goths(x):x\in\calX^\cl\}$ is closed in $\mathbb R$, we deduce that $f$ is a continuous surjection from $\calX^{\Berk}$ to the finite set $\{\goths(x):x\in\calX^\cl\}$. Then $f$ is constant on each connected component of $\calX^{\mathrm{Berk}}$, and thus the number of connected components of $\calX^{\mathrm{Berk}}$ (hence of $\calX$) is at least $\#(\{\goths(x):x\in\calX^{\cl}\})$.
	\end{proof}
	
	We will show in Proposition~\ref{prop-crystabelline-slopes} that the slope
	set $\{\goths(x):x\in\calX^{\cl}\}$ is \emph{finite}. Corollary~\ref{cor-naive-lower-bound} then implies that the number
	of distinct crystabelline slopes occurring in $\calX$ gives a \emph{naive
		lower bound} for the Breuil--M\'ezard multiplicity of $\Rkis$. Moreover, we can use slopes to separate $\Min(\Rkis)$ as follows. Note that for $\gothp\in\Min(\Rkis)$, the \emph{slope $\goths(\p)$}, namely, the slope of some (equivalently, any) closed point of $\Spf(\Rkis/\p)^{\rig}$ is well-defined. Then we have the following result.
	
	\begin{corollary}
		\label{rmk-1}
		Write $R:=\Rkis$.
		Assume that $\{\goths(x):x\in\calX^{\cl}\}$ is finite. Then we have  the following decomposition
		\[
		\Min(R)=\bigsqcup_{\goths\in\{\goths(x):x\in\calX^{\cl}\}}\Sigma_{\goths},\quad\textrm{where}\quad\Sigma_{\goths}=\{\p\in\Min(R):\goths(\p)=\goths\}.
		\]
		Moreover, for each slope
		$\goths$ occurring in $\calX$, set
		\[
		\gotha_\goths
		:=\bigcap_{\substack{x\in\calX^{\cl}\\\goths(x)=\goths}}
		(\gothm_x\cap \Rkis),
		\qquad
		R_\goths:=\Rkis/\gotha_\goths,
		\]
		where $\gothm_x$ is the maximal ideal of $\Rkis[1/\vp]$ corresponding
		to $x$. Then $\gotha_\goths=\cap_{\p\in\Sigma_{\goths}}\p$, and a closed point $x$ of
		$\Spec \Rkis[1/\vp]$ lies in $\Spec R_\goths[1/\vp]$ if and only if
		$\goths(x)=\goths$.
	\end{corollary}
	\begin{proof}
		By Corollary~\ref{cor-naive-lower-bound}, the slope is constant
		on each irreducible component of $\calX$, which gives the asserted partition of $\Min(R)$. Moreover, $R[1/\vp]$ is Jacobson by
		\cite[Lemma~15.10.5]{stack}, so each of its minimal primes
		is the intersection of the maximal ideals containing it. Now this corollary is clear.
	\end{proof}
	\subsection{Newton polygons of crystabelline slopes}
	\label{subsection-NP}
	We now apply the results of \cite[\S~7]{Lghost2} to compute the
	crystabelline slopes occurring in $\calX$. Let
	$\Delta\simeq\FF_p^\times$ denote the torsion subgroup of
	$\ZZ_p^\times$.
	
	\begin{notation}
		\label{notation-1}
		For the fixed wildly ramified type $\t=\theta_1\oplus\theta_2$ of
		conductor $m\geq2$ and Hodge--Tate weights $k_1<k_2$, introduce the
		following notation.
		\begin{enumerate}
			\item Define the character
			\[
			\e:=\bar\theta_1|_\Delta\cdot\omega^{-k_2}
			\times\bar\theta_2|_\Delta\cdot\omega^{-k_1-1}
			\]
			of $\begin{pmatrix}\Delta&0\\0&\Delta\end{pmatrix}$. We assume that $\e$ is \emph{relevant} to $\sigma=\sigma_{a,b}$, so that
			$
			\e=\omega^{-s_\e+b}\times\omega^{s_\e+a+b}$ for some
			$s_\e\in\{0,\dots,p-2\}$.
			This assumption is necessary for
			$R_{\bar\rho}^{\square,\psi}(\w,\t)$ to be nonzero. In particular, we have
			\begin{equation}
				\label{eq-1}
				\bar\theta_1|_\Delta\cdot\omega^{-k_2}=\omega^{-s_\e+b},
				\qquad
				\bar\theta_2|_\Delta\cdot\omega^{-k_1-1}=\omega^{s_\e+a+b}.
			\end{equation}
			
			\item The character $\e$ is also relevant to
			$\sigma^{\co}=\sigma_{p-3-a,a+b+1}$, so that $\e=\omega^{-s'_\e+b'}\times\omega^{s'_\e+a'+b'}$, where
			$
			a'=p-3-a$, $b'=a+b+1$, and $s'_\e=\{s_\e+a+1\}$.
			\item Set
			$
			w_*:=\theta_2\theta_1^{-1}(\exp(p))
			\exp\bigl(p(k_2-k_1-1)\bigr)-1\in\gothm_{\CC_p}$.
			Then $v_p(w_*)=1/(p^{m-2}(p-1))$.
		\end{enumerate}
	\end{notation}
	
	Let $G_\sigma^{(\e)}(w,t)\in\calO[w][\![t]\!]$ be the \emph{ghost
		series} of \cite[Definition~2.5]{Lghost2}. Roughly speaking, the
	results of \cite[\S~7]{Lghost2} relate the slopes of
	$\NP(G_\sigma^{(\e)}(w,-))$ to the values $v_p(\delta_1(p))$
	at points $(\delta_1,\delta_2,\rho_x)$ of the \emph{trianguline
		deformation space} $\calX_{\bar\rho}^{\square,\tri}$ with the
	corresponding weight parameter $w=(\delta_1\delta_2^{-1}\chi_{\cycl}^{-1})(\exp(p))-1$. In particular, specializing to $w=w_*$ allows us to compute the crystabelline
	slopes of $\calX$. Proposition~\ref{prop-crystabelline-slopes}
	gives the precise relation between these slopes and the slopes of
	$\NP(G_\sigma^{(\e)}(w_*,-))$ in $[0,k_2-k_1]$.
	
	We do not recall the definition of the ghost series. Instead,
	we give an explicit description of its Newton polygon at $w_*$.
	
	\begin{definition}
		\label{def-halo-slopes}
		For $\e=\omega^{-s_\e+b}\times\omega^{s_\e+a+b}$, set
		$
		\eta_1:=\min\{s_\e,\{s_\e+a\}\}$, and
		$\eta_2:=\max\{s_\e,\{s_\e+a\}\}$.
		Define a sequence $(a_n)_{n\geq1}$ by
		\[
		a_{2\ell+1}:=\eta_1+\ell(p-1),
		\quad\textrm{and}\quad
		a_{2\ell+2}:=\eta_2+\ell(p-1)
		\qquad(\ell\geq0).
		\]
		For $n\geq1$, the $n$th \emph{Hodge ratio} is
		$
		\lambda_n^{(\e)}:=a_n-\lfloor a_n/p\rfloor$.
		The \emph{Newton ratios} $\goths_n^{(\e)}$ are defined by the
		following formulas, for $n\geq0$.
		If $a+s_\e<p-1$, set
		\begin{equation}
			\goths_{n+1}^{(\e)}-\lambda_{n+1}^{(\e)}
			=\begin{cases}
				1,& n-2s_\e\equiv1,3,\dots,2a+1\pmod{2p},\\
				-1,& n-2s_\e\equiv2,4,\dots,2a+2\pmod{2p},\\
				0,&\text{otherwise}.
			\end{cases}
		\end{equation}
		If $a+s_\e\geq p-1$, set
		\begin{equation}
			\goths_{n+1}^{(\e)}-\lambda_{n+1}^{(\e)}
			=\begin{cases}
				1,& n-2s_\e\equiv2,4,\dots,2a+2\pmod{2p},\\
				-1,& n-2s_\e\equiv3,5,\dots,2a+3\pmod{2p},\\
				0,&\text{otherwise}.
			\end{cases}
		\end{equation}
	\end{definition}
	
	\begin{proposition}
		\label{prop-halo-slopes}
		For $w_*$ as above and $n\geq1$, the $n$th slope of
		$\NP(G_\sigma^{(\e)}(w_*,-))$ is
		$
		\goths_n^{(\e)}\cdot v_p(w_*)
		=\goths_n^{(\e)}/(p^{m-2}(p-1))$.
		The sequence $(\goths_n^{(\e)})_{n\geq1}$ is strictly increasing,
		and $\goths_n^{(\e)}=0$ if and only if $s_\e=0$ and $n=1$.
	\end{proposition}
	
	\begin{proof}
		This follows from  \cite[Definition--Proposition~2.13]{Lghost2}.
	\end{proof}
	
	\begin{remark}
		\begin{enumerate}
			\item We have
			$
			\NP(G_{\sigma_{a,b}}^{(\e)}(w,-))
			=\NP\bigl(G_{\sigma_{a,0}}^{(\e\cdot(\omega^{-b}\times\omega^{-b}))}(w,-)\bigr)$; see \cite[Remark~2.30]{Lghost}.
			Thus the determinant twist $\omega^b$ does not enter the formulas in
			Definition~\ref{def-halo-slopes} and Proposition~\ref{prop-halo-slopes}.
			
			\item 	Up to the possible slope zero (of multiplicity one) part, $\NP(G_{\sigma}^{(\e)}(w,-))$ coincides with $\NP(G_{\sigma^{\co}}^{(\e)}(w,-))$; see \cite[Proposition~2.15, Remark~2.16]{Lghost2}. Write $\e_0=\omega^{-s'_\e}\times\omega^{s'_\e+a'}$. By (1), we also have $$\NP(G^{(\e)}_{\sigma_{p-3-a,a+b+1}}(w,-))=\NP(G^{(\e\cdot(\omega^{-(a+b+1)}\times\omega^{-(a+b+1)}))}_{\sigma_{p-3-a,0}}(w,-))=\NP(G^{(\e_0)}_{\sigma_{p-3-a,0}}(w,-)).$$
		\end{enumerate}
	\end{remark}
	
	The following proposition is largely contained in \cite[\S~7]{Lghost2}.
	We give a detailed treatment of the \emph{boundary slopes}
	$0$ and $k_2-k_1$, particularly when $\bar\rho$ is split.
	Our argument relies on the local ghost theorem
	\cite[Theorem~2.6]{Lghost2}, so we assume $p\geq11$.
	\begin{proposition}
		\label{prop-crystabelline-slopes}
		Assume that $p\geq11$.
		Let $\calX=\bigl(\Spf R_{\bar\rho}^{\square,\psi}(\w,\t)\bigr)^{\rig}$
		and $k_0:=\{k_2-k_1-1\}$.
		\begin{enumerate}
			\item A slope $\goths\in(0,k_2-k_1)$ occurs in $\calX$ if and
			only if it is a slope of $\NP(G_\sigma^{(\e)}(w_*,-))$.
			
			\item Suppose that $\bar\rho$ is non-split. Then:
			\begin{enumerate}
				\item the slope $0$ occurs in $\calX$ if and only if $s_\e=0$;
				\item the slope $k_2-k_1$ occurs in $\calX$ if and only if
				$k_0=\{s_\e+a\}$.
			\end{enumerate}
			
			\item Suppose that $\bar\rho$ is split. Then:
			\begin{enumerate}
				\item the slope $0$ occurs in $\calX$ if and only if $s_\e=0$
				or $s'_\e=\{s_\e+a+1\}=0$;
				\item the slope $k_2-k_1$ occurs in $\calX$ if and only if
				$k_0=\{s_\e+a\}$ or $k_0=\{s'_\e+a'\}=\{s_\e-1\}$.
			\end{enumerate}
		\end{enumerate}
	\end{proposition}
	
\begin{proof}
	(1) follows from \cite[Corollary~7.8]{Lghost2}.
	We first prove (2). Write
	\[
	\bar\rho\simeq
	\begin{pmatrix}
		\unr(\bar\b_1)\omega^{a+b+1}&*\neq0\\
		0&\unr(\bar\b_2)\omega^b
	\end{pmatrix}
	=:\begin{pmatrix}\bar\chi_1&*\neq0\\0&\bar\chi_2\end{pmatrix}.
	\]
	For each closed point $x\in\calX$,  $\rho_x$
	has the triangulation \eqref{eq:trangulation}, with parameters
	\[
	\delta_1=x^{-k_1}(\theta_2,\a_2),
	\quad\textrm{and}\quad
	\delta_2=x^{-k_2}(\theta_1,\a_1).
	\]
	
	If $0$ is a crystabelline slope of $\calX$, say $\goths(x)=0$, then $v_p(\delta_1(p))=0$. In this case $\delta_1$ is a subrepresentation of $\rho_x$ and hence  $\bar\theta_2|_{\Delta}\cdot\omega^{-k_1}\equiv\bar\delta_1|_{\Delta}\equiv\omega^{a+b+1}\pmod{\varpi}$, which implies that $s_{\e}=0$ by \eqref{eq-1}. If $k_2-k_1$ is a crystabelline slope of $\calX$, then $(\delta_1,\delta_2,\infty)\in\scrS_+^\ord$ as we proved in Lemma~\ref{lem-crystabelline-slopes}. Hence \cite[Proposition~4.13]{colmez-triangulline} provides another triangulation of the corresponding representation $\rho_x$, namely $(\delta_1'=(x^{-k_1})(\theta_1,\a_1),\delta_2'=(x^{-k_2})(\theta_2,\a_2),\infty)$. In this case, $v_p(\delta_1'(p))=0$ and similarly we have
	$\bar\theta_1|_{\Delta}\cdot\omega^{-k_1}\equiv\bar\delta'_1|_{\Delta}\equiv\omega^{a+b+1}\pmod{\varpi}$, which implies that $k_0=\{s_\e+a\}$. This proves one direction of (a) and (b).
	
	For the other direction of (2), we directly construct the required crystabelline lifts. (a) If $s_\e=0$, then consider the crystabelline characters for $i=1,2$,  $\chi_i:=\chi_{\cycl}^{-k_i}\cdot(\theta_{3-i},\b_i)$, where $\b_i\in\calO^\times$ is a lift of $\bar\b_i$ and $\chi_1\chi_2=\psi\chi_{\cycl}$ (this is possible because $(\psi\chi_{\cycl}^{k_1+k_2+1})|_{\I_{\QQ_p}}=\theta_1\theta_2$). We first construct a closed point on $\calX$. Note that $s_{\e}=0$ ensures that $\chi_i$ has mod-$\varpi$ reduction $\bar\chi_i$. Since $\bar\rho$ is the unique non-split extension (up to isomorphism) of $\bar\chi_2$ by $\bar\chi_1$, any non-split extension in $\Ext^1_{G_{\QQ_p}}(\chi_2,\chi_1)$ admits a $G_{\QQ_p}$-invariant $\calO$-lattice whose mod-$\varpi$ reduction is $\bar\rho$.  Note that $\theta_1\neq\theta_2$ implies that $\Ext^1_{G_{\QQ_p}}(\chi_2,\chi_1)$ is one dimensional. Moreover, each such non-split extension is indeed crystabelline, since both $\chi_i$ are crystabelline and the Hodge--Tate weight of $\chi_1\chi_2^{-1}$ is at most $-1$ (note that $\theta_1\neq\theta_2$ excludes the non-crystabelline extension). Hence we obtain a closed point $x$ of $\calX$, and we next calculate its slope. Let $\a_1$ be the eigenvalue of the $\phi$-action on the subspace of $\mathbb{D}_{\operatorname{pcrys}}(\rho_x)$ where $I_{\QQ_p}$ acts through $\theta_1$. Then weak admissibility gives $v_p(\a_1)=k_2$, which implies that $\goths(x)=k_2-v_p(\a_1)=0$. (b) If $k_0=\{s_\e+a\}$, we take crystabelline characters $\chi_i':=\chi_{\cycl}^{-k_i}\cdot(\theta_{i},\b_{i})$ with the same $\b_i$, so that $\chi_1'\chi_2'=\psi\chi_{\cycl}$, and a similar argument provides a closed point $x'$ of $\calX$ that is an extension of $\chi_2'$ by $\chi_1'$. In this case, $v_p(\a_1')=k_1$ and hence the slope $\goths(x')=k_2-v_p(\a'_1)=k_2-k_1$. This completes the proof of (2).
	
	We now prove (3). The ``if" part follows from (2), as for a non-split representation $\rho_x$ constructed in (2), we may choose a $G_{\QQ_p}$-invariant $\calO$-lattice (by conjugating $\begin{pmatrix}
		\vp^n&0\\0&1
	\end{pmatrix}$ for some $n\in\ZZ$), such that its mod-$\varpi$ reduction is split. For the other direction, let $x$ be a closed point of $\calX$ with slope $\goths(x)=0$ (resp. $k_2-k_1$). If $\rho_x$ is non-split, then $\delta_1$ (resp. $\delta_1'$) in the second paragraph is a subrepresentation of $\rho_x$; hence its mod-$\vp$ reduction is one of $\{\bar\chi_1,\bar\chi_2\}$ (note that $\delta_1\equiv\bar\chi_i$ if and only if $\delta_2\equiv\bar\chi_{3-i}$ by comparing the determinants and a similar equivalence holds for $\delta'_i$).  Comparing the power of $\omega$ yields the required equalities in (a)(b). If $\rho_x$ is split, then $\mathbb{D}_{\operatorname{pcrys}}(\rho_x)$ is also split. Hence both $\delta_1$ (resp. $\delta'_1$) and $\delta_2$ (resp. $\delta_2'$) are subrepresentations of $\rho_x$ by checking weak admissibility. Similarly we also obtain the required equalities in this case. We complete the proof of (3).
\end{proof}
	\begin{remark}
		Combining this proposition with Lemma~\ref{lem-crystabelline-slopes}, we observe that the reducible locus of $\calX$ is nonempty if and only if $s_{\e}=0$ or $k_0=\{s_\e+a\}$ for non-split $\bar\rho$, and $s_{\e}=0$, or $k_0=\{s_\e+a\}$, or $s'_{\e}=0$, or $k_0=\{s'_\e+a'\}$ for split $\bar\rho$. See \cite[Lemma~6.1]{hu-paskunas} for a similar statement of the split case.
		
		Let $\bar\rho_1$ and $\bar\rho_2$ be the two non-split extensions, whose associated Serre weights are $\sigma=\sigma_{a,b}$ and $\sigma^{\co}=\sigma_{p-3-a,a+b+1}$, respectively. Let $\bar\rho_0$ be the semisimplification of $\bar\rho_i$. Denote by $\calX_i$ the corresponding Kisin deformation space for $i=0,1,2$. From the above proof, the boundary slopes $0$ and $k_2-k_1$ may occur simultaneously in $\calX_i$. However, for $\goths\in\{0,k_2-k_1\}$, it occurs in at most one of $\{\calX_1,\calX_2\}$, since
		our genericity assumption excludes $a=p-2$. See Corollary~\ref{cor-split-compare-mult} for more details.
	\end{remark}
	
	\subsection{Counting crystabelline slopes}
	\label{subsection-counting-slopes}In this subsection, we assume that $p\geq11$.
	By Proposition~\ref{prop-crystabelline-slopes}, the crystabelline
	slopes occurring in $\calX$ form a \emph{finite} set. Thus
	Corollary~\ref{cor-naive-lower-bound} gives the following \emph{naive lower bound} for the Breuil--M\'ezard multiplicity $e(\Rkis/\vp)$
	\[
	e'(\calX):=e'\bigl(\Spf(\Rkis)^{\rig}\bigr)
	:=\#\{\goths(x):x\in\calX^{\cl}\}.
	\]
	We now compute this bound explicitly using
	Definition~\ref{def-halo-slopes}. The main result for non-split
	$\bar\rho$ is the following proposition; the split case is
	treated in Corollary~\ref{cor-split-compare-mult}.
	
	\begin{proposition}
		\label{prop-BM-mult=halo-mult}Assume that $p\geq11$.
		If $\bar\rho$ is non-split, then we have
		$$
		e(\Rkis/\vp)=e'\bigl(\Spf(\Rkis)^{\rig}\bigr).$$
	\end{proposition}
	
	\begin{proof}
		Write $e:=e(\Rkis/\vp)$. Since the crystabelline slopes of $\calX$ lie in $[0,k_2-k_1]$ and each slope of $\NP(G^{(\e)}_\sigma(w_*,-))$ has multiplicity one (Proposition~\ref{prop-halo-slopes}), our strategy is to compute the $e$th and $(e+1)$th slopes of $\NP(G^{(\e)}_\sigma(w_*,-))$ and compare them with $k_2-k_1$. Combining this comparison with a separate analysis of the
		boundary slopes $0$ and $k_2-k_1$, it follows from
		Proposition~\ref{prop-crystabelline-slopes} that the first $e$ slopes of $\NP(G^{(\e)}_\sigma(w_*,-))$ are exactly all of the slopes occurring in $\calX$. This implies that $e=e'$.
		
		We first fix some notation. Recall that $\e=\bar\theta_1|_{\Delta}\cdot \omega^{-k_2}\times\bar\theta_2|_{\Delta}\cdot \omega^{-k_1-1}=\omega^{-s_\e+b}\times\omega^{s_\e+a+b}$, and that $w_*$ (see Notation~\ref{notation-1}) has valuation $v_p(w_*)=1/(p^{m-2}(p-1))$. Write $\bar\theta_i|_{\Delta}=\omega^{n_i}$ for $i=1,2$. Put
		\begin{itemize}
			\item 
			$e=e(\Rkis/\vp)$, and $e'=e'(\Spf(\Rkis)^\rig)$.
			\item $k:=k_2-k_1+1$, $k_0=\{k-2\}$, and $k-2=k_0+k_\bullet(p-1)$.
			\item $C:=C(m,k)=(k-1)(p-1)p^{m-2}$, and $N=N(m,k)=k_\bullet+(k-1)\frac{p^{m-1}-1}{p-1}+1$.
			\item $\eta_1=\min\{s_\e,\{s_\e+a\}\}$, and $\eta_2=\max\{s_\e,\{s_\e+a\}\}$. 
		\end{itemize}
		Note that $C\cdot v_p(w_*)=k-1=k_2-k_1$.
		Since $\bar\theta_1|_\Delta\cdot\omega^{-k_2}=\omega^{-s_\e+b}$,
		Proposition~\ref{prop-halo-mult} gives
		\begin{equation}
			\label{eq-pf-1}
			\begin{aligned}
				e&=\bigl[\bar{\sigma}(\mathbf{w},\tau)^{\mathrm{ss}}:\sigma\bigr]=2(k-1)\frac{p^{m-1}-1}{p-1}+\left\lfloor\frac{k-2-s_\e}{p-1}\right\rfloor+\left\lfloor\frac{k-2-\{s_\e+a\}}{p-1}\right\rfloor+2
				\\&=2N+\left\lfloor\frac{k_0-s_\e}{p-1}\right\rfloor+\left\lfloor\frac{k_0-\{s_\e+a\}}{p-1}\right\rfloor=2N+
				\begin{cases}
					-2& 0\leq k_0<\eta_1,\\
					-1& \eta_1\leq k_0<\eta_2,\\
					0&k_0\geq\eta_2.
				\end{cases}
			\end{aligned}
		\end{equation}
		Here the last equality follows from the formulas
		\[
		\left\lfloor\frac{k_0-s_\e}{p-1}\right\rfloor
		=
		\begin{cases}
			-1,& k_0<s_\e,\\
			0,& k_0\geq s_\e,
		\end{cases}\quad\textrm{ and }\quad
		\left\lfloor\frac{k_0-\{s_\e+a\}}{p-1}\right\rfloor
		=
		\begin{cases}
			-1,& k_0<\{s_\e+a\},\\
			0,& k_0\geq \{s_\e+a\}.
		\end{cases}
		\]
		Moreover, since $2N\equiv 2(k+k_\bullet)\equiv2(k_0+2)\pmod{2p}$, we obtain
		\[
		e\equiv\begin{cases}
			2k_0+2,& 0\leq k_0<\eta_1,\\
			2k_0+3,& \eta_1\leq k_0<\eta_2,\\
			2k_0+4,& k_0\geq\eta_2.
		\end{cases}\pmod{2p}.
		\]
		Using Definition~\ref{def-halo-slopes} and Proposition~\ref{prop-halo-slopes}, we compute the differences $\goths_e^{(\e)}-C$ and $\goths_{e+1}^{(\e)}-C$ in each case. The results are summarized in the following table. 
		\begin{table}[H]
			\centering
			\renewcommand{\arraystretch}{2.5}
			\setlength{\tabcolsep}{5pt}
			\resizebox{\textwidth}{!}{
				\begin{tabular}{c|c|cc|c|cc|c}
					\hline
					\multicolumn{2}{c|}{Case}
					& $\lambda_e^{(\e)}-C$
					& $\goths_e^{(\e)}-C$
					& $e-2s_\e\pmod{2p}$
					& $\lambda_{e+1}^{(\e)}-C$
					& $\goths_{e+1}^{(\e)}-C$
					& $e+1-2s_\e\pmod{2p}$
					\\
					\hline
					\multirow{6}{*}{$s_\e+a<p-1$}
					& $k_0<\eta_1-1$
					& $\eta_2-k_0-(p-1)$
					& $\eta_2-k_0-(p-1)<0$
					& $2(k_0-s_\e)+2$
					& $\eta_1-k_0-1$
					& $\eta_1-k_0-1>0$
					& $2(k_0-s_\e)+3$
					\\
					&$\substack{
						k_0=\eta_1-1 \quad \text{and}\quad \eta_1\geq1,\\[4pt]
						\text{or } k_0=p-2 \quad \text{and}\quad \eta_1=0
					}$
					& $\eta_2-\eta_1-p+2$
					& $\eta_2-\eta_1-p+2<0$
					& $0$
					& $0$
					& $0$
					& $1$
					\\
					& $\eta_1\leq k_0<\eta_2$
					& $\eta_1-k_0$
					& $\eta_1-k_0-1<0$
					& $2(k_0-s_\e)+3$
					& $\eta_2-k_0-1$
					& $\eta_2-k_0>0$
					& $2(k_0-s_\e)+4$
					\\
					& $k_0=\eta_2=\{s_\e+a\}$
					& $0$
					& $0$
					& $2a+4$
					& $p-2+\eta_1-\eta_2$
					& $p-2+\eta_1-\eta_2>0$
					& $2a+5$
					\\
					& $\eta_2<k_0<\eta_1+p-2$
					& $\eta_2-k_0$
					& $\eta_2-k_0<0$
					& $2(k_0-s_\e)+4$
					& $p-2+\eta_1-k_0$
					& $p-2+\eta_1-k_0>0$
					& $2(k_0-s_\e)+5$
					\\
					\hline
					\multirow{5}{*}{$s_\e+a\geq p-1$}
					& $k_0<\eta_1$
					& $\eta_2-k_0-(p-1)$
					& $\eta_2-k_0-p<0$
					& $2(k_0-s_\e)+2$
					& $\eta_1-k_0-1$
					& $\eta_1-k_0>0$
					& $2(k_0-s_\e)+3$
					\\
					& $k_0=\eta_1=\{s_\e+a\}$
					& $0$
					& $0$
					& $2a+5$
					& $\eta_2-\eta_1-1$
					& $\eta_2-\eta_1-1>0$
					& $2a+6$
					\\
					& $\eta_1<k_0<\eta_2-1$
					& $\eta_1-k_0$
					& $\eta_1-k_0<0$
					& $2(k_0-s_\e)+3$
					& $\eta_2-k_0-1$
					& $\eta_2-k_0-1>0$
					& $2(k_0-s_\e)+4$
					\\
					& $k_0=\eta_2-1=s_\e-1$
					& $\eta_1-\eta_2+1$
					& $\eta_1-\eta_2+1<0$
					& $1$
					& $0$
					& $0$
					& $2$
					\\
					& $\eta_2\leq k_0$
					& $\eta_2-k_0$
					& $\eta_2-k_0-1<0$
					& $2(k_0-s_\e)+4$
					& $p-2+\eta_1-k_0$
					& $p-1+\eta_1-k_0>0$
					& $2(k_0-s_\e)+5$
					\\
					\hline
				\end{tabular}
			}
			\caption{Hodge and Newton ratios relative to $C$ at indices $e$ and $e+1$.}
			\label{tab-halo-slopes-at-e}
		\end{table}
		Combining Propositions~\ref{prop-halo-slopes} and~\ref{prop-crystabelline-slopes} with the preceding table, we obtain the following summary of the first, $e$th, and $(e+1)$th slopes of $\NP(G^{(\e)}_\sigma(w_*,-))$ and the range of crystabelline slopes of $\calX$. In the last column, a square bracket at $0$ or $k-1$ indicates that the corresponding endpoint occurs as a crystabelline slope of $\calX$.
		
		\begin{center}
			\renewcommand{\arraystretch}{1.5}
			\begin{tabular}{cc|ccc|c}
				\hline
				\multicolumn{2}{c|}{\multirow{2}{*}{Conditions}}
				& \multicolumn{3}{c|}{Slopes of $\NP(G^{(\e)}_\sigma(w_*,-))$}
				& \multirow{2}{*}{Counting interval} \\
				\cline{3-5}
				& & $1$st & $e$th & $(e+1)$th & \\
				\hline
				$s_\e\neq0$ & $k_0\neq\{s_\e+a\}$
				& $>0$ & $<k-1$ & $\geq k-1$ & $(0,k-1)$ \\
				$s_\e=0$ & $k_0\neq\{s_\e+a\}$
				& $0$ & $<k-1$ & $\geq k-1$ & $[0,k-1)$ \\
				$s_\e\neq0$ & $k_0=\{s_\e+a\}$
				& $>0$ & $k-1$ & $>k-1$ & $(0,k-1]$ \\
				$s_\e=0$ & $k_0=\{s_\e+a\}$
				& $0$ & $k-1$ & $>k-1$ & $[0,k-1]$ \\
				\hline
			\end{tabular}
		\end{center}
		In each case, by Proposition~\ref{prop-crystabelline-slopes}, we observe that the first $e$ slopes of $\NP(G^{(\e)}(w_*,-))$ are exactly all of the slopes occurring in $\calX$, which implies that $e=e'$.
	\end{proof}
	
	\begin{remark}
		\label{rmk-2}
		The equality in Proposition~\ref{prop-BM-mult=halo-mult} relies
		on two features of our setting: $\bar\rho$ has a unique Serre
		weight, and the principal series type is wildly ramified.
		The latter condition implies that $w_*$ lies in the \emph{halo region} of
		\cite{Lghost2}, where every slope of
		$\NP(G_\sigma^{(\e)}(w_*,-))$ has multiplicity one.
		
		When $\bar\rho$ is split, it has two Serre weights, and the
		equality $e=e'$ no longer holds. As we will see in
		Corollary~\ref{cor-split-compare-mult}, $e$ is approximately
		twice $e'$.
		If $\tau$ is a tamely ramified principal series type, the corresponding parameter $w'_*$ lies in the central region of the weight disc. In this case, the ``near-Steinberg'' phenomenon gives rise to long segments of constant slope in $\NP(G^{(\e)}_\sigma(w'_*,-))$, and usually we have $e'<e$; see \cite[Proposition~2.17(5)]{Lghost2}. The relationship between $e$ and $e'$ in this case is more subtle, and we leave its investigation to future work.
	\end{remark}
	
	For split $\bar\rho$, Proposition~\ref{prop-crystabelline-slopes}
	allows us to compute the naive lower bound from the two
	corresponding non-split cases. To compare this bound with the
	Breuil--M\'ezard multiplicity, we first compare the multiplicities $e(\Rkis/\vp)$ for the two non-split $\bar\rho$'s with Serre weight sets $\{\sigma_{a,b}\}$ and $\{\sigma_{p-3-a,a+b+1}\}$, respectively.
	\begin{lemma}
		\label{lem-multiplicity-difference}
		Let $e_1$ and $e_2$ be the multiplicities $e(\Rkis/\vp)$ for the
		two non-split residual representations with Serre weight sets
		$\{\sigma_{a,b}\}$ and $\{\sigma_{p-3-a,a+b+1}\}$, respectively.
		Recall that $k_0=\{k_2-k_1-1\}$ and
		$s_\e\in\{0,\dots,p-2\}$ is defined in Notation~\ref{notation-1}.
		Set $s'_\e:=\{s_\e+a+1\}$ and $a':=p-3-a$.
		Then $e_1-e_2$ is given by the following table.
		
		\begin{center}
			\renewcommand{\arraystretch}{1.2}
			\setlength{\tabcolsep}{2pt}
			\begin{tabular}{c|l}
				\hline
				$e_1-e_2$ & Conditions \\
				\hline
				$2$
				& $s_\e=0,\quad k_0=\{s_\e+a\}$
				\\
				\hline
				\multirow{2}{*}{$1$}
				& $s_\e=0,\quad	k_0\notin\bigl\{\{s'_\e+a'\},\{s_\e+a\}\bigr\}$
				\\
				& $s_\e\neq0,\quad s'_\e\neq 0,\quad k_0=\{s_\e+a\}$
				\\
				\hline
				\multirow{3}{*}{$0$}
				& $s_\e=0,\quad k_0=\{s'_\e+a'\}$\\
				&$s'_\e=0,\quad k_0=\{s_\e+a\}$
				\\
				& $s_\e\neq0,\quad s'_\e\neq0,\quad
				k_0\notin\bigl\{\{s'_\e+a'\},\{s_\e+a\}\bigr\}$
				\\
				\hline
				\multirow{2}{*}{$-1$}
				& $s'_\e=0,\quad k_0\notin\bigl\{\{s'_\e+a'\},\{s_\e+a\}\bigr\}$
				\\
				& $s_\e\neq0,\quad s'_\e\neq 0,\quad k_0=\{s'_\e+a'\}$
				\\
				\hline
				$-2$
				& $s'_\e=0,\quad k_0=\{s'_\e+a'\}$
				\\
				\hline
			\end{tabular}
		\end{center}
		
	\end{lemma}
	
	\begin{proof}
		By Proposition~\ref{prop-halo-mult},
		\[
		\begin{aligned}
			e_1-e_2
			&=\left\lfloor\frac{k_0-s_\e}{p-1}\right\rfloor
			+\left\lfloor\frac{k_0-\{s_\e+a\}}{p-1}\right\rfloor-\left\lfloor\frac{k_0-\{s_\e+a+1\}}{p-1}\right\rfloor
			-\left\lfloor\frac{k_0-\{s_\e-1\}}{p-1}\right\rfloor.
		\end{aligned}
		\]
		Note that $s'_\e=\{s_\e+a+1\}$ and
		$\{s'_\e+a'\}=\{s_\e-1\}$. For $0\leq r,k_0\leq p-2$,
		\[
		\left\lfloor\frac{k_0-r}{p-1}\right\rfloor
		-\left\lfloor\frac{k_0-\{r+1\}}{p-1}\right\rfloor
		=\mathbf1_{k_0=r}-\mathbf1_{r=p-2}.
		\]
		Applying this identity with $r=\{s_\e+a\}$ and
		$r=\{s_\e-1\}$ gives the required differences.
	\end{proof}
	Recall that, by our genericity assumption, the conditions
	$s_\e=0$ and $s'_\e=0$ are mutually exclusive, as are the conditions
	$k_0=\{s_\e+a\}$ and $k_0=\{s'_\e+a'\}$.
	\begin{corollary}
		\label{cor-split-compare-mult}Assume that $p\geq11$.
		Retain the notation of Lemma~\ref{lem-multiplicity-difference},
		and suppose that $\bar\rho$ is split. Recall that $k=k_2-k_1+1$ and $k_0=\{k_2-k_1-1\}$. Write
		$
		\calX:=\Spf(\Rkis)^{\rig}$, $e:=e(\Rkis/\vp)$, and $e':=e'(\calX)$.
		Then the following statements hold and cover all possible cases.
		\begin{enumerate}
			\item If $s_\e\neq0$, $s'_\e\neq0$, and
			$
			k_0\notin\bigl\{\{s_\e+a\},\{s'_\e+a'\}\bigr\}$,
			then neither $0$ nor $k-1$ occurs as a crystabelline
			slope of $\calX$, and
			$				
			e=2e'.
			$
			\item If exactly one of $s_\e,s'_\e$ is zero  and
			$
			k_0\notin\bigl\{\{s_\e+a\},\{s'_\e+a'\}\bigr\},
			$
			then $0$ is the only boundary slope occurring in
			$\calX$, and $e+1=2e'$.
			
			\item If $s_\e\neq0$, $s'_\e\neq0$, and
			$
			k_0\in\bigl\{\{s_\e+a\},\{s'_\e+a'\}\bigr\}$,
			then $k-1$ is the only boundary slope occurring in
			$\calX$, and $e+1=2e'$.
			
			\item If exactly one of $s_\e,s'_\e$ is zero and
			$
			k_0\in\bigl\{\{s_\e+a\},\{s'_\e+a'\}\bigr\},
			$
			then both $0$ and $k-1$ occur as crystabelline
			slopes of $\calX$, and
			$
			e+2=2e'$.
		\end{enumerate}
	\end{corollary}
	
	\begin{proof}
		Let $e_1$, $e_2$ be the Breuil--M\'ezard multiplicities defined in Lemma~\ref{lem-multiplicity-difference}, and denote by $e'_1$, $e'_2$ the corresponding naive lower bounds. From Proposition~\ref{prop-BM-mult=halo-mult} and Lemma~\ref{lem-type-Serre-wt}(3), we obtain $e_i=e'_i$ for $i=1,2$ and $e=e_1+e_2$. By Proposition~\ref{prop-crystabelline-slopes}, we have $e'=\max\{e'_1,e_2'\}$, except for $s_\e=0$ and $k_0=\{s'_{\e}+a'\}$, or $s'_\e=0$ and $k_0=\{s_{\e}+a\}$, in which cases $e'=e_1'+1=e_2'+1$.
		Now (1)(2)(3)(4) follow from Proposition~\ref{prop-crystabelline-slopes} and  Lemma~\ref{lem-multiplicity-difference}. 
	\end{proof}
	\section{Breuil--M\'ezard multiplicity of components}
	\label{section-computing-BM-mult-of-components}
	In this section, we prove the main results of this paper.
	The non-split case (Theorem~\ref{thm-BM=1-non-split}) follows directly
	from Proposition~\ref{prop-BM-mult=halo-mult}.
	For the split case (Theorem~\ref{thm-main-thm-split}), we first use the
	results on \emph{versal} deformation rings in
	\cite{paskunas-2adic,HT-BMconj} to compare irreducible components
	in the split and non-split cases; see
	\S~\ref{subsection-analysis-of-comp-1}
	and \ref{subsection-analysis-of-comp-2}.
	This yields upper bounds on the number of irreducible components
	and on their individual Breuil--M\'ezard multiplicities; see Propositions~\ref{prop-key-irreducible},~\ref{prop-key-reducible} and Corollaries~\ref{cor-component-injection-1}, \ref{cor-component-injection-2}. In \S~\ref{finish-the-proof},
	we determine the number of components
	by matching the resulting upper bound with the naive lower bound $e'(\calX)$ of Corollary~\ref{cor-split-compare-mult}. Similarly, we then show that the \emph{total} Breuil--M\'ezard multiplicity agrees with the sum of the upper bounds for the individual
	multiplicities, so each bound is attained.
	
	Throughout this section, we assume that $p\geq11$.
	\subsection{Breuil--M\'ezard multiplicity one or two}
	We first consider a non-split $\bar\rho$.  The relation $e(\Rkis/\vp)=e'(\calX)$ (Proposition~\ref{prop-BM-mult=halo-mult}) implies the following result, namely, each irreducible component of $\Rkis$ has Breuil--M\'ezard multiplicity one.
	\begin{theorem}
		\label{thm-BM=1-non-split}
		Assume that $p\geq11$.
		Let $\bar\rho$ be a reducible, non-split, and very generic local representation. Let $\t$ be a wildly ramified principal series type. Write $\Min(\Rkis)=\{\p_1,\dots,\p_n\}$, and $R(\p_i):=\Rkis/\p_i$.
		Then $$n=e(\Rkis/\vp)\quad \textrm{and}\quad e(R(\p_i)/\vp)=1 \quad\textrm{for}\quad i\in\{1,\dots,n\}.$$ Moreover, $R(\p_i)$ is a formal power series ring over $\calO$.
	\end{theorem} 
	\begin{proof}
		Write $\calX=\Spf(\Rkis)^\rig$.
		By Corollary~\ref{cor-naive-lower-bound}, Proposition~\ref{prop-crystabelline-slopes}, and Proposition~\ref{prop-BM-mult=halo-mult}, we have
		\[
		e(R_{\bar\rho}^{\square,\psi}(\w,\t)/\vp)\geq n\geq e'(\calX)=e(R_{\bar\rho}^{\square,\psi}(\w,\t)/\vp).
		\]
		Hence $n=e(R_{\bar\rho}^{\square,\psi}(\w,\t)/\vp)$. Then Lemma~\ref{lem-inequality-of-mult} implies that $e(R(\p_i)/\vp)=1$ for any $i$. The last assertion follows from Lemma~\ref{lem comm alg2}.
	\end{proof}
	
	Recall that the refined
	decomposition of $\Rkis$ described in Corollary~\ref{rmk-1} gives the following partition of the set of minimal primes $\Min(\Rkis)$.
	\begin{definition}
		\label{def-decomp-of-Min(R)}
		Write $R=\Rkis$.
		We define \begin{equation}
			\label{eq-decomp-of-Min(R)}
			\gotha_\goths
			:=
			\bigcap_{\substack{x\in\calX^{\cl}\\ \goths(x)=\goths}}
			(\gothm_x
			\cap R)\quad\textrm{and}\quad\Sigma_{\goths}(\bar\rho):=\{\p\in\Min(R):\gotha_\goths\subseteq\p\}.
		\end{equation} Then we have the decomposition $\Min(R_{\bar\rho}^{\square,\psi}(\w,\t))=\bigsqcup_{\goths\in\{\goths(x):x\in\calX^\cl\}}\Sigma_\goths(\bar\rho)$. For simplicity, we also use $\Sigma_{(0,k_2-k_1)}(\bar\rho)$ to denote the union $\cup_{\goths\in(0,k_2-k_1)}\Sigma_{\goths}(\bar\rho)$.
	\end{definition}
	In the very special case where $\bar\rho$ is non-split, each $\gotha_\goths$ is a minimal prime and each irreducible component $\Spec \Rkis/\gotha_\goths$ is uniquely determined by the slope $\goths$.
	For split $\bar\rho$, our main result is the following, in which the Breuil--M\'ezard multiplicity of an irreducible component of $\Rkis$ can be one or two.
	\begin{theorem}
		\label{thm-main-thm-split}
		Assume that $p\geq11$.
		Keep the notation of Definition~\ref{def-decomp-of-Min(R)}.
		Let $\bar\rho$ be a reducible, split, and very generic local representation. Let $\t$ be a wildly ramified principal series type. Write $R=\Rkis$, and let $e'(\Spf(R)^\rig)$ denote the number of distinct crystabelline slopes of $\Spf(R)^\rig$. Then $\Spec R$ has exactly $e'(\Spf(R)^\rig)$ irreducible components. In particular, each $\gotha_\goths$ in the decomposition \eqref{eq-decomp-of-Min(R)} is minimal.
		
		Moreover, we have the following statements about Breuil--M\'ezard multiplicities. For $\gothp\in\Min(R)$, we write $R(\p):=R/\p$.
		\begin{enumerate}
			\item If $\p\in\Sigma_\goths(\bar\rho)$ for some $\goths\in(0,k_2-k_1)$, then $e(R(\p)/\vp)=2$.
			\item If $\p\in\Sigma_\goths(\bar\rho)$ for some $\goths\in\{0,k_2-k_1\}$, then $e(R(\p)/\vp)=1$ and $R(\p)$ is a formal power series ring over $\calO$.
		\end{enumerate}
	\end{theorem}
	The proof of this theorem will occupy the remaining part of this section and will be completed in \S~\ref{finish-the-proof}.
	\begin{remark}
		\label{rmk-7}
		Let $\bar\rho$ be a reducible, very generic representation, and let $\tau$ be a principal series type. As indicated in Remark~\ref{rmk-2}, the Breuil--M\'ezard multiplicities of the irreducible components of $\Rkis$ can exhibit more complicated behavior than in the special case where $\bar\rho$ is non-split and $\t$ is wildly ramified.
		We treat the case where $\bar\rho$ is split and $\tau$ is wildly ramified in this paper. The passage from the non-split case to the split case relies primarily on Galois deformation theory and does not require the full strength of the ghost series results of \cite{Lghost}.
		The tamely ramified case is the subject of ongoing work. 
	\end{remark}
	
	\subsection{Various local deformation rings}
	\label{subsection-versal-deformation-ring}
	To relate the Kisin deformation rings for non-split $\bar\rho$ to those for split $\bar\rho$, it is more convenient to consider the \emph{unframed} version. We introduce various Kisin deformation rings and discuss their relations in this subsection.
	\begin{notation}
		In the remainder of this paper, we consider the following reducible and very generic representations. 	
		Write $\bar\chi_1:=\unr(\bar\b_1)\omega^{a+b+1}$ and $\bar\chi_2=\unr(\bar\b_2)\omega^b$, where $\bar\b_i\in\FF^\times$ and $a\in\{2,\dots,p-5\}$. Put
		\[
		\bar\rho_0=\bar\chi_1\oplus\bar\chi_2,\quad
		\bar\rho_1\simeq \begin{pmatrix}\bar\chi_1&*\neq0\\0&\bar\chi_2\end{pmatrix},\quad\textrm{and}\quad	\bar\rho_2\simeq \begin{pmatrix}\bar\chi_2&*\neq0\\0&\bar\chi_1\end{pmatrix}.
		\]
	\end{notation}
	We first recall from \cite{paskunas-2adic} a few constructions of (pseudo-)deformation rings. As noted in \cite[Remark~7.4]{paskunas-2adic}, these constructions also apply in the \emph{fixed determinant} setting; see also \cite[\S~3.2]{HT-BMconj}. Throughout the remainder of this section, we omit the superscript $\psi$.
	
	For $i=0,1,2$, let $R^{\ps}:=R^{\ps,\psi}_{\tr\bar\rho_i}$ be the universal pseudo-deformation ring of $\tr\bar\rho_i$ with fixed determinant $\psi\chi_{\cycl}$. For $i=1,2$, denote by $R_i:=R_{\bar\rho_i}^{\psi}$ the universal (unframed) deformation ring of $\bar\rho_i$ with fixed determinant $\psi\chi_{\cycl}$, with universal lift $\rho_i^{\univ}$. Note that taking traces induces isomorphisms (\cite[Proposition~3.6]{paskunas-2adic})
	$
	t_i:R^{\mathrm{ps}}\simeq R_{i}$.
	Let $\mathfrak{r}$ be the reducibility ideal of $R^{\mathrm{ps}}$ defined in \cite[Proposition~3.8]{paskunas-2adic}. By that proposition, it is a principal ideal generated by a regular element $c\in\gothm_{R^{\ps}}$. 
	Recall the following versal ring of $\bar\rho_0$ defined in \cite[\S~5]{paskunas-2adic} $$R^{\ver}:=R_{\bar\rho_0}^{\ver,\psi}=\frac{R^{\ps,\psi}[[x,y]]}{xy-c}.$$
	Let $
	\rho^{\ver}:G_{\QQ_p}\rightarrow\operatorname{CH}(R^{\ps})^\times\rightarrow\GL_2(R^{\ver})
	$ be the versal deformation of $\bar\rho_0$ therein, where $\operatorname{CH}(R^{\ps})$ is the Cayley--Hamilton algebra of $R^{\ps}$.
	
	Now we come to the constructions of \emph{versal Kisin deformation rings} from \cite[\S~6]{paskunas-2adic}. Let $\bar\rho$ be either
	$\bar\rho_1$, $\bar\rho_2$, or $\bar\rho_0$, and let $R$ be $R_1$, $R_2$, or $R^{\mathrm{ver}}$, respectively. Fix a $p$-adic Hodge type $(\w,\t,\psi)$.
	Define $\Spec R(\w,\t)$ to be the closure in $\Spec R$ of the set of closed points $x\in\Spec R[1/\vp]$ that are potentially crystalline of type $(\w,\t,\psi)$, equipped with the reduced scheme structure. Hence $R(\w,\t)$ is reduced and $\calO$-torsion free, characterized by the property that a closed point $x\in\Spec R[1/p]$ lies in $\Spec R(\w,\t)[1/p]$ if and only if the corresponding representation $\rho_x$ is of type $(\w,\t,\psi)$. By \cite[Remark~7.4]{paskunas-2adic}, the framed Kisin deformation rings $R_{\bar\rho_1}^{\square,\psi}(\w,\t)$, $R_{\bar\rho_2}^{\square,\psi}(\w,\t)$, and $R_{\bar\rho_0}^{\square,\psi}(\w,\t)$ are formally smooth over the versal Kisin deformation rings $R_1(\w,\t)$, $R_2(\w,\t)$, and $R^{\ver}(\w,\t)$, respectively.
	\begin{lemma}
		\label{lem-1}
		Let $A$ be one of $R_1(\w,\t)$, $R_2(\w,\t)$, or
		$R^{\ver}(\w,\t)$, and let
		$A^\square=R_{\bar\rho}^{\square,\psi}(\w,\t)$
		be the corresponding framed Kisin ring.
		Then $A$ is a reduced, equidimensional, complete Noetherian
		local $\calO$-flat algebra, and $A[1/\vp]$
		is regular. Moreover,
		$
		e(A/\vp)=e(A^\square/\vp)$,
		and extension of ideals induces a bijection
		\begin{equation}\label{eq-pf-2}
			\Min(A)\xrightarrow{\sim}\Min(A^\square):
			\qquad
			\p\longmapsto\p A^\square.
		\end{equation}
		For every $\p\in\Min(A)$, we also have
		$
		e\bigl(A/(\p,\vp)\bigr)
		=
		e\bigl(A^\square/(\p A^\square,\vp)\bigr)$.	
	\end{lemma}
	
	\begin{proof}
		By construction, $A$ is a reduced complete Noetherian local
		$\calO$-flat algebra.
		The map $A\to A^\square$ is formally smooth and induces an
		isomorphism on residue fields, hence 
		$
		A^\square\simeq A[[X_1,\ldots,X_n]]$.
		Consequently,
		$
		A^\square/\vp
		\simeq (A/\vp)[[X_1,\ldots,X_n]]$, and this gives
		$e(A/\vp)=e(A^\square/\vp)$.
		
		For every prime ideal $\p$ of $A$, we have
		$
		A^\square/\p A^\square
		\simeq (A/\p)[[X_1,\ldots,X_n]]$.
		Thus $\p A^\square$ is prime and
		$\p A^\square\cap A=\p$.
		If $\p$ is minimal and $\mathfrak q\subseteq\p A^\square$
		is a prime ideal, then $\mathfrak q\cap A=\p$, which forces
		$\p A^\square\subseteq\mathfrak q$.
		Hence $\p A^\square$ is minimal, and
		\eqref{eq-pf-2} is well-defined and injective.
		Conversely, if $\mathfrak q\in\Min(A^\square)$, going-down
		for the flat map $A\to A^\square$
		\cite[Lemma~10.11]{eis} implies that
		$\p:=\mathfrak q\cap A$ is a minimal prime of $A$.
		Since $\p A^\square$ is a prime ideal contained in
		$\mathfrak q$, minimality gives
		$\mathfrak q=\p A^\square$.
		This proves the bijection.
		
		By Theorem~\ref{thm-Kisin}, $A^\square$ is equidimensional.
		For every $\p\in\Min(A)$, the preceding bijection gives
		$
		\dim(A/\p)
		=
		\dim(A^\square/\p A^\square)-n
		=
		\dim A^\square-n
		=
		\dim A$.
		Thus $A$ is equidimensional.
		Moreover,
		$
		A^\square/(\p A^\square,\vp)
		\simeq
		\bigl(A/(\p,\vp)\bigr)[[X_1,\ldots,X_n]]$
		gives the asserted equality
		for the special fibers of corresponding irreducible components.
		
		Finally, $A\to A^\square$ is a flat local homomorphism,
		hence faithfully flat \cite[Lemma~10.39.17]{stack}.
		Its base change
		$
		A[1/\vp]\longrightarrow A^\square[1/\vp]$
		is also faithfully flat by
		\cite[Lemma~10.39.7]{stack}.
		Since $A^\square[1/\vp]$ is regular by
		Theorem~\ref{thm-Kisin}, faithful flat descent of regularity
		\cite[Lemma~10.164.4]{stack} shows that $A[1/\vp]$
		is regular.
	\end{proof}
	\begin{remark}
		\label{rmk-3}
		Via the bijection
		$\Min(A)\simeq\Min(A^\square)$ in Lemma~\ref{lem-1},
		the decomposition of $\Min(A^\square)$ according to crystabelline
		slopes in Definition~\ref{def-decomp-of-Min(R)} induces
		a corresponding decomposition of $\Min(A)$. Indeed, the representations corresponding to closed points of the framed
		and unframed generic fibers are isomorphic, as the framing records
		only a choice of basis. Consequently, the crystabelline slope
		calculations apply equally to the unframed generic fiber. In particular, when $\t$ is wildly ramified, Theorem~\ref{thm-BM=1-non-split} holds for the unframed rings $R_i(\w,\t)=R_{\bar\rho_i}^{\psi}(\w,\t)$.
	\end{remark}

	\begin{remark}
		Alternatively, one can work directly on $\Spf(A)^\rig$ by
		repeating the argument of Proposition~\ref{prop-slope-map}
		with $(\calX,\calV)$ replaced by
		$\bigl(\Spf(A)^\rig,\rho'|_{\Spf(A)^\rig}\bigr)$,
		where $\rho'$ denotes $\rho_1^{\univ}$, $\rho_2^{\univ}$, or
		$\rho^{\ver}$, according to the choice of $A$.
		The resulting slope function agrees with that on the framed
		generic fiber after pullback, and the same calculations yield
		the decomposition of $\Min(A)$.
	\end{remark}
	From now on, we work with versal Kisin deformation rings and assume that $\t$ is a wildly ramified principal series type.
	We will establish the conclusion of Theorem~\ref{thm-main-thm-split}
	for $R^{\ver}(\w,\t)$.
	\subsection{Irreducible and reducible loci of Kisin rings}
	\label{subsection-analysis-of-comp-1}
	Let $R\in\{R_1,R_2,R^{\ver}\}$.
	We recall a classification of closed points of $\Spec R[1/\vp]$ from \cite[\S~6]{paskunas-2adic}, which is used to describe the relations between $R_{i}(\w,\t)$ and $R^{\ver}(\w,\t)$ in \cite[Lemma~6.4]{paskunas-2adic}. 
	Let $x$ be a closed point of $\Spec R[1/\vp]$, which is potentially
	crystalline of type $(\mathbf{w},\tau,\psi)$. Then $\rho_x$ is reducible
	if and only if the reducibility ideal $\mathfrak r$ is mapped to zero under the composition
	$
	R^{\mathrm{ps}}\longrightarrow R\xrightarrow{x^*}E'$, where $E'$ is the residue field of $x$.
	If
	$\rho_x$ is reducible, then it is an extension
	$$
	0\longrightarrow \kappa_1\longrightarrow \rho_x
	\longrightarrow \kappa_2\longrightarrow 0,$$
	where
	$
	\kappa_1,\kappa_2:G_{\mathbb{Q}_p}\longrightarrow
	\mathcal{O}_{E'}^\times$
	are potentially crystalline characters. 
	Moreover, after possibly relabeling $\kappa_1$ and $\kappa_2$, we may assume that the Hodge--Tate weight of $\kappa_i$ is $k_i$ for $i=1,2$, which holds automatically if the
	extension is non-split. 
	\begin{definition}
		\label{def-reducible-type}
		Let $x$ be a closed point of $\Spec R(\w,\t)[1/\vp]$. If $\rho_x$ is irreducible, then we say $x$ is of \emph{reducibility type} $\irr$. Otherwise, we say that $x$ is of
		reducibility type $1$ (resp. $2$) if
		$
		\kappa_1\equiv\bar\chi_1\pmod{\vp_{E'}}$ (resp. $
		\kappa_1\equiv\bar\chi_2\pmod{\vp_{E'}}$), where $ \vp_{E'}$ generates the maximal ideal of $\mathcal O_{E'}$.
		For $*\in\{1,2,\mathrm{irr}\}$, we define
		$I_*^{\mathrm{ver}}$ to be the ideal of $R^{\mathrm{ver}}$ and
		$I_*^{\mathrm{ps}}$ to be the ideal of $R^{\mathrm{ps}}$ given by
		\[
		I_*^{\mathrm{ver}}
		:=
		R^{\mathrm{ver}}\cap\bigcap_x\mathfrak{m}_x,
		\qquad
		I_*^{\mathrm{ps}}
		:=
		R^{\mathrm{ps}}\cap I_*^{\mathrm{ver}},
		\]
		where the intersection is taken over all closed points
		$
		x\in\operatorname{Spec}
		R^{\mathrm{ver}}(\mathbf{w},\tau)[1/p]
		$
		of reducibility type $*$.  In particular, $c\in I_1^{\ps}\cap I_2^{\ps}$.
	\end{definition}
	Using these ideals $I_?^?$, we may relate the Kisin deformation rings for
	non-split $\bar\rho$ to those for split $\bar\rho$. The following result is one of the main inputs of our proof.
	\begin{proposition}
		\label{prop:kisin rings relation}
		Recall that we have isomorphisms $R^{\mathrm{ver}}
		=
		R^{\mathrm{ps}}[\![x,y]\!]/(xy-c)$, and $R^{\mathrm{ps}}\xrightarrow{\sim}R_i$ for $i=1,2$.
		Then we have $I_{\irr}^{\ver}=I_{\irr}^{\ps}R^{\ver}$, $I_1^{\ver}=I_1^{\ps}R^{\ver}+(x)$, and $I_2^{\ver}=I_2^{\ps}R^{\ver}+(y)$. Moreover, the above isomorphisms induce the following isomorphisms.
		\begin{align*}
			&R^{\mathrm{ps}}/
			\bigl(I_i^{\mathrm{ps}}\cap I_{\mathrm{irr}}^{\mathrm{ps}}\bigr)
			\cong R_i(\mathbf{w},\tau),\quad\textrm{and} \quad
			R^{\mathrm{ver}}(\mathbf{w},\tau)\simeq R^{\ver}/(I^{\ver}_{\irr}\cap I^{\ver}_{1}\cap I^{\ver}_{2}).\\
			&(R^{\mathrm{ps}}/I_1^{\mathrm{ps}})[\![y]\!]
			\cong
			R^{\mathrm{ver}}/I_1^{\mathrm{ver}},
			\qquad
			(R^{\mathrm{ps}}/I_2^{\mathrm{ps}})[\![x]\!]
			\cong
			R^{\mathrm{ver}}/I_2^{\mathrm{ver}},
			\quad \textrm{and}\quad
			\frac{(R^{\mathrm{ps}}/I_{\mathrm{irr}}^{\mathrm{ps}})[\![x,y]\!]}{(xy-c)}
			\cong
			R^{\mathrm{ver}}/I_{\mathrm{irr}}^{\mathrm{ver}}.
		\end{align*}
	\end{proposition}
	\begin{proof}
		See \cite[Lemmas~6.4, 6.5, and 6.6]{paskunas-2adic} and their proofs. 
	\end{proof}
	
	\begin{remark}
		\label{rmk-4}
		By Proposition~\ref{prop-crystabelline-slopes} and its proof, we have the following observations. A type $1$ (resp. type $2$) point occurs in $\Spec R^{\ver}(\w,\t)[1/\vp]$ if and only if 
		\begin{equation}
			\label{eq-3}
			\begin{aligned}
				\omega^{a+b+1}\textrm{ (resp. }\omega^b\textrm{)}\in\{\bar\theta_1|_{\Delta}\cdot\omega^{-k_1},\bar\theta_2|_{\Delta}\cdot\omega^{-k_1}\}.
			\end{aligned}
		\end{equation}
		Moreover, the following table summarizes equivalent conditions to \eqref{eq-3}.
		In each row, the conditions in the last two columns are equivalent
		to each other and to the existence of a type $i\in\{1,2\}$ point on $\Spec R^{\ver}(\w,\t)[1/\vp]$ of slope $\goths\in\{0,k_2-k_1\}$,
		with the type and slope specified in the first column.
		\begin{table}[H]
			\centering
			\renewcommand{\arraystretch}{1.4}
			\begin{tabular}{|c|c|c|}
				\hline
				Type and slope & Character condition & Numerical condition \\
				\hline
				Type $1$, slope $k_2-k_1$
				& $\bar\theta_1|_{\Delta}\cdot\omega^{-k_1}=\omega^{a+b+1}$
				& $k_0=\{s_\e+a\}$ \\
				\hline
				Type $1$, slope $0$
				& $\bar\theta_2|_{\Delta}\cdot\omega^{-k_1}=\omega^{a+b+1}$
				& $s_\e=0$ \\
				\hline
				Type $2$, slope $k_2-k_1$
				& $\bar\theta_1|_{\Delta}\cdot\omega^{-k_1}=\omega^b$
				& $k_0=\{s'_\e+a'\}=\{s_\e-1\}$ \\
				\hline
				Type $2$, slope $0$
				& $\bar\theta_2|_{\Delta}\cdot\omega^{-k_1}=\omega^b$
				& $s'_\e=\{s_\e+a+1\}=0$ \\
				\hline
			\end{tabular}
		\end{table}
		Again, by our genericity assumption, for each
		$\goths\in\{0,k_2-k_1\}$, type $1$ and type $2$ points of
		slope $\goths$ cannot both occur on
		$\Spec R^{\ver}(\w,\t)[1/\vp]$.
	\end{remark}
	
	Recall the decomposition of $\Min(R^{\ver}(\w,\t))$ defined in Definition~\ref{def-decomp-of-Min(R)}, namely, $$\Min(R^{\ver}(\w,\t))=\Sigma_{(0,k_2-k_1)}(\bar\rho_0)\sqcup\Sigma_0(\bar\rho_0)\sqcup\Sigma_{k_2-k_1}(\bar\rho_0).$$ We summarize the relations between this decomposition and the ideals $I^{\ver}_{*}$ of Definition~\ref{def-reducible-type} in the following lemma.
	\begin{lemma}
		\label{lem-relation-slope-and-reducible-type}
		Let $A=R^{\ver}(\w,\t)$, and let $\gothp\in\Min(A)$. We also use $I_*^{\ver}$ to denote the image of this ideal in $A$.  
		\begin{enumerate}
			\item 
			We have  $\p\in V(I_{\irr}^{\ver})$ if and only if $\p\in\Sigma_{(0,k_2-k_1)}(\bar\rho_0)$, in which case $c\notin\p$.
			\item 
			If $\p\in\Sigma_\goths(\bar\rho_0)$ for $\goths\in\{0,k_2-k_1\}$, then we have $\p\in V(I_i^{\ver})\setminus V(I_{3-i}^{\ver})$, where $i\in\{1,2\}$ is given by
			\begin{equation}
				\label{eq-4}
				i=\begin{cases}
					1,&\textrm{if } \goths=0,\textrm{ and } s_\e=0,\\
					2,&\textrm{if } \goths=0,\textrm{ and } s'_\e=0,\\
					1,&\textrm{if } \goths=k_2-k_1,\textrm{ and } \{s_\e+a\}=\{k_2-k_1-1\},\\
					2,&\textrm{if } \goths=k_2-k_1,\textrm{ and } \{s'_\e+a'\}=\{k_2-k_1-1\}.
				\end{cases}
			\end{equation}
		\end{enumerate}
	\end{lemma}
	\begin{proof}
		For (1), the ring $A[1/\vp]$ is Jacobson
		by \cite[Lemma~15.10.5]{stack}, and
		Lemma~\ref{lem-crystabelline-slopes} gives
		\begin{equation}
			\label{eq-5}
			\begin{aligned}
				I_{\irr}^{\ver}
				= A\cap
				\bigcap_{\rho_x\text{ irreducible}}\mathfrak m_x 
				= A\cap
				\bigcap_{\p\in\Sigma_{(0,k_2-k_1)}(\bar\rho_0)}\p A[1/\vp]
				= \bigcap_{\p\in\Sigma_{(0,k_2-k_1)}(\bar\rho_0)}\p.
			\end{aligned}
		\end{equation}
		Since $A$ has finitely many minimal primes, it follows that for every $\p\in\Min(A)$,
		$
		I_{\irr}^{\ver}\subseteq\p$ if and only if $
		\p\in\Sigma_{(0,k_2-k_1)}(\bar\rho_0)$.
		For such a prime $\p$, any closed point $x$ of $V(\p A[1/\vp])$ corresponds to an irreducible representation $\rho_x$, hence
		$c\notin\mathfrak m_x$. This implies that $c\notin\p$.
		
		For (2), let $\p\in\Sigma_{\goths}(\bar\rho_0)$, where
		$\goths\in\{0,k_2-k_1\}$.
		By our genericity assumption and Remark~\ref{rmk-4},
		all closed points of $\Spec A[1/\vp]$ of slope $\goths$ have the
		same type $i$, as specified in \eqref{eq-4}.
		In particular, every closed point of $V(\p A[1/\vp])$ has type $i$ and slope $\goths$.
		Since $A[1/\vp]$ is Jacobson, we obtain
		$$I^{\ver}_{i}=\bigcap_{x\textrm{ of type } i}(\gothm_x\cap A)\subseteq\bigcap_{\substack{x\textrm{ of type } i\\\goths(x)=\goths}}(\gothm_x\cap A)\subseteq\bigcap_{\substack{
				x\textrm{ of type } i\\
				\goths(x)=\goths\\
				\p\subseteq\gothm_x\cap A
		}}(\gothm_x\cap A)=\p.$$
		It remains to show that $I_{3-i}^{\ver}\not\subseteq\p$.
		Set $\goths'=k_2-k_1-\goths$.
		As we mentioned above, no type $3-i$ point has slope
		$\goths$, and thus every type $3-i$ point has slope $\goths'$.
		Consequently,
		\[
		\begin{aligned}
			I_{3-i}^{\ver}
			&= \bigcap_{x\text{ of type }3-i}
			(\mathfrak m_x\cap A) \supseteq
			\bigcap_{\goths(x)=\goths'}
			(\mathfrak m_x\cap A)
			= \bigcap_{\p'\in\Sigma_{\goths'}(\bar\rho_0)}\p'.
		\end{aligned}
		\]
		Suppose that $I_{3-i}^{\ver}\subseteq\p$. 
		If $\Sigma_{\goths'}(\bar\rho_0)$ is empty,  then $A=I_{3-i}^{\ver}$, and hence the preceding inclusion gives $A=I_{3-i}^{\ver}\subseteq\p$, a contradiction.
		Otherwise,
		there is some
		$\p'\in\Sigma_{\goths'}(\bar\rho_0)$ such that $\p'=\p$, which is a contradiction by comparing the slopes.
		Thus $I_{3-i}^{\ver}\not\subseteq\p$, as required.
	\end{proof}
	\begin{remark}
		The proof of \cite[Lemma~6.4]{paskunas-2adic} shows that every
		closed point $x$ of the generic fiber whose maximal ideal contains
		$I_{3-i}^{\ver}$ is of type $3-i$.
		This immediately implies that $\p\notin V(I_{3-i}^{\ver})$,
		since every closed point of $\Spec((A/\p)[1/\vp])$ has type $i$.
		We have included the more detailed argument above to make the
		relationship between reducibility types and slopes explicit.
	\end{remark}
	\subsection{Analysis of irreducible components}
	\label{subsection-analysis-of-comp-2}
	Let $\gothp$ be a minimal prime of $R^{\ver}(\mathbf{w},\tau)=R^{\ver}/(I^{\ver}_{\irr}\cap I^{\ver}_1\cap I^{\ver}_2)$, and write $R^{\ver}(\gothp):=R^{\ver}(\mathbf{w},\tau)/\gothp$.
	In what follows, we will relate $\p$ to a minimal prime $\tilde{\gothp}$ of $R_i(\w,\t)$ using Proposition~\ref{prop:kisin rings relation}. We identify prime ideals of these quotient rings with their inverse images in $R^{\ver}$ or $R^{\ps}$. For instance, we also view $\p$ as a point of $\Spec R^{\ver}$
	lying in $V(I^{\ver}_{\irr}\cap I^{\ver}_1\cap I^{\ver}_2)$.
	
	We first consider $\p\in\Sigma_{(0,k_2-k_1)}(\bar\rho_0)$.  By Proposition~\ref{prop:kisin rings relation}, we have the following ring homomorphisms (for $i=1,2$)
	\begin{equation}
		\label{eq-pf-3}
		\begin{aligned}
			R_i(\mathbf{w},\tau)
			&\simeq R^{\ps}/(I^{\ps}_{\irr}\cap I^{\ps}_i)
			\twoheadrightarrow R^{\ps}/I^{\ps}_{\irr}
			\longrightarrow
			\frac{(R^{\ps}/I^{\ps}_{\irr})[\![x,y]\!]}{(xy-c)}
			\\
			&\simeq R^{\ver}/I_{\irr}^{\ver}
			\twoheadleftarrow
			R^{\ver}/(I^{\ver}_{\irr}\cap I^{\ver}_{1}\cap I^{\ver}_{2})
			\simeq R^{\ver}(\mathbf{w},\tau).
		\end{aligned}
	\end{equation}
	By Lemma~\ref{lem-relation-slope-and-reducible-type}, $\p\in V(I_{\irr}^{\ver})$ is a minimal prime of $R^{\ver}/I^{\ver}_{\irr}$ and $c\notin\p$. We denote the preimage of $\p$ in $R_i(\w,\t)$ by $\tilde{\p}_i$ under the above map. Clearly $\tilde{\p}_i\in V(I_{\irr}^{\ps})$ is a prime ideal of $R_{i}(\w,\t)$. Moreover, we also have $\tilde{\p}_i\notin V(I_i^{\ps})$ since $c\in (I_{1}^{\ps}\cap I_2^{\ps})\setminus \p$.
	\begin{proposition}
		\label{prop-key-irreducible}
		The prime ideal $\tilde\gothp_i$ is minimal and hence corresponds to an irreducible component of $R_i(\mathbf{w},\tau)$. In particular, $R_i(\mathbf{w},\tau)/\tilde{\gothp}_i$ is a formal power series ring over $\calO$. Moreover, we have the following isomorphism
		\[
		(R_i(\mathbf{w},\tau)/\tilde{\gothp}_i)[\![x,y]\!]/(xy-\bar c)\simeq R^{\ver}(\gothp)=R^{\ver}(\w,\t)/\p,
		\]
		where $\bar c$ is the image of $c$ in $R^{\ps}/\tp_i$, and  $e(R^{\ver}(\gothp)/\varpi)\leq2$.
	\end{proposition}
	\begin{proof}
		From \eqref{eq-pf-3}, we have the following homomorphisms
		\[
		R_i(\mathbf{w},\tau)
		\simeq R^{\ps}/(I^{\ps}_{\irr}\cap I^{\ps}_i)
		\twoheadrightarrow R^{\ps}/I^{\ps}_{\irr}
		\longrightarrow
		\frac{(R^{\ps}/I^{\ps}_{\irr})[\![x,y]\!]}{(xy-c)}
		\simeq R^{\ver}/I_{\irr}^{\ver}.
		\]
		By the corollary after \cite[Theorem~22.6]{mat} (see also \cite[Lemma~6.2]{paskunas-2adic}), the homomorphism $R^{\ps}/I_{\irr}^{\ps}\rightarrow R^{\ver}/I_{\irr}^{\ver}$ is flat (taking $A=R^{\ps}/I^{\ps}_{\irr}$, $B=A[\![x,y]\!]$, $f(x,y)=xy-c$). Hence going-down for a flat map (\cite[Lemma~10.11]{eis}) implies that the prime ideal $\tilde\gothp_i$ is minimal over $I^{\ps}_{\irr}$. Moreover, by Lemma~\ref{lem-minimal-primes}(2), $\tilde\p_i\notin V(I_{i}^{\ps})$ implies that $\tp_i$ is also a minimal prime ideal over $I_{\irr}^{\ps}\cap I_i^{\ps}$, and hence a minimal prime of $R_i(\mathbf{w},\tau)
		\simeq R^{\ps}/(I^{\ps}_{\irr}\cap I^{\ps}_i)$.
		Therefore, we deduce that $R^{\ps}/\tp_i=R_i(\w,\t)/\tp_i$ is a formal power series ring by Theorem~\ref{thm-BM=1-non-split} and Remark~\ref{rmk-3}.
		Next,  consider the following homomorphisms obtained by reducing modulo $\tp_i$
		\[
		C:=\frac{(R^{\ps}/\tilde\gothp_i)[[x,y]]}{(xy-\bar c)}
		\simeq
		\frac{R^{\ps}[[x,y]]/(xy-c)}
		{\tilde\gothp_i\bigl(R^{\ps}[[x,y]]/(xy-c)\bigr)}= \frac{R^{\ver}(\mathbf{w},\tau)}{\tilde\gothp_iR^{\ver}(\w,\t)}
		\twoheadrightarrow R^{\ver}(\mathbf{w},\tau)/\gothp.
		\]
		Since $c\notin\tp_i$, we have that $C$ is an integral domain, which implies that $\tilde\gothp_i R^{\ver}(\mathbf{w},\tau)$ is a prime ideal. This forces $\tilde\gothp_i R^{\ver}(\mathbf{w},\tau)\subseteq\gothp$ to be an equality, since $\gothp$ is a minimal prime ideal of $R^{\ver}(\mathbf{w},\tau)$. Hence we obtain the required isomorphism. The last inequality $e(R^{\ver}(\gothp)/\varpi)\leq2$ then follows from Lemma~\ref{lem:mult leq 2}.
	\end{proof}
	\begin{corollary}
		\label{cor-component-injection-1}
		For $i=1,2$, there is an injection from $\Sigma_{(0,k_2-k_1)}(\bar\rho_0)$ to $\Sigma_{(0,k_2-k_1)}(\bar\rho_i)$.
	\end{corollary}
	\begin{proof}
		This follows directly from Proposition~\ref{prop-key-irreducible}, where the map is given by $\p\mapsto\tilde{\p}_i$. Note that this map is injective since $\tp_iR^{\ver}(\w,\t)=\p$.
	\end{proof}
	
	We next deal with minimal primes $\p\in\Sigma_0(\bar\rho_0)\cup\Sigma_{k_2-k_1}(\bar\rho_0)$. 
	Fix such a minimal prime and assume that $\p\in\Sigma_\goths(\bar\rho_0)$, where $\goths\in\{0,k_2-k_1\}$. By Lemma~\ref{lem-relation-slope-and-reducible-type}(1)(2), we have $\p\in V(I_i^{\ver})\setminus(V(I_{3-i}^{\ver})\cup V(I_{\irr}^{\ver}))$ for some $i\in\{1,2\}$ uniquely determined by $\p$. We slightly change the notation of $x,y$ in the presentation of $R^{\ver}$, and we set $z_1=y$ and $z_2=x$. Hence $z_i\in I_{3-i}^{\ver}$ for $i=1,2$. Then from Proposition~\ref{prop:kisin rings relation}, we have $z_{3-i}\in\p$, and the following homomorphisms
	\begin{align*}
		R_i(\mathbf{w},\tau)&\simeq R^{\ps}/(I_i^{\ps}\cap I^{\ps}_{\irr}) \twoheadrightarrow R^{\ps}/I_i^{\ps}\rightarrow (R^{\ps}/I_i^{\ps})[\![z_i]\!]
		\\
		&\simeq R^{\ver}/I_i^{\ver}\twoheadleftarrow R^{\ver}/(I_1^{\ver}\cap I_2^{\ver}\cap I^{\ver}_{\irr})=\Rvv.
	\end{align*}
	Let $\tp_i$ be the preimage of the minimal prime $\p$ under the above map $R_i(\w,\t)\rightarrow R^{\ver}/I_i^{\ver}$. Clearly, we know $\tp_i\in V(I_i^{\ps})$. Moreover, using the relations $I^{\ps}_{\irr}R^{\ver}=I^{\ver}_{\irr}$ from Proposition~\ref{prop:kisin rings relation} and $\p\notin V(I^{\ver}_{\irr})$, we deduce that $\tp_i\notin V(I^{\ps}_{\irr})$.
	\begin{proposition}
		\label{prop-key-reducible}
		The prime ideal $\tilde\gothp_i$ is minimal and hence corresponds to an irreducible component of $R_i(\mathbf{w},\tau)$. In particular, $R_i(\mathbf{w},\tau)/\tilde{\gothp}_i$ is a formal power series ring over $\calO$. Moreover, we have the following isomorphism
		\[
		(R_i(\mathbf{w},\tau)/\tilde{\gothp}_i)[\![z_i]\!]\simeq R^{\ver}(\gothp)=R^{\ver}(\w,\t)/\p.
		\]
		In particular, $R^{\ver}(\p)$ is a formal power series ring and $e(R^{\ver}(\gothp)/\varpi)=1$.
	\end{proposition}
	\begin{proof}
		The proof is similar to that of
		Proposition~\ref{prop-key-irreducible}.
		Using the flatness of the map
		$
		R^{\ps}/I_i^{\ps}
		\rightarrow
		(R^{\ps}/I_i^{\ps})[\![z_i]\!]
		$
		and going-down, we deduce that $\tp_i/I_i^{\ps}$ is a
		minimal prime of $R^{\ps}/I_i^{\ps}$.
		Moreover,
		$\tp_i\in V(I_i^{\ps})\setminus V(I_{\irr}^{\ps})$,
		together with Lemma~\ref{lem-minimal-primes}(2), implies
		that  $\tp_i$ is a minimal prime ideal of
		$
		R_i(\w,\t)\simeq
		R^{\ps}/(I_i^{\ps}\cap I_{\irr}^{\ps})
		$. Hence
		$
		R^{\ps}/\tp_i
		\simeq R_i(\w,\t)/\tp_i
		$
		is a formal power series ring over $\calO$ by
		Theorem~\ref{thm-BM=1-non-split} and
		Remark~\ref{rmk-3}.
		Next, reducing modulo $\tp_i$ gives
		\[
		C:=(R^{\ps}/\tp_i)[\![z_i]\!]
		\simeq
		\frac{R^{\ver}/I_i^{\ver}}
		{\tp_i(R^{\ver}/I_i^{\ver})}
		\simeq
		\frac{R^{\ver}(\w,\t)}{\tp_i R^{\ver}(\w,\t)+(z_{3-i})}
		\twoheadrightarrow R^{\ver}(\w,\t)/\p.
		\]
		Here the second isomorphism follows from the relations $I_i^{\ver}=I_i^{\ps}R^{\ver}+(z_{3-i})$ and $
		I_i^{\ps}\subseteq\tp_i$, and the last map follows from
		$\tp_i \Rvv\subseteq\p$ and $z_{3-i}\in\p$.
		Since $C$ is an integral domain, indeed a formal power
		series ring over $\calO$, the ideal
		$\tp_i \Rvv+(z_{3-i})$ is prime.
		The minimality of $\p$ therefore gives
		$
		\p=\tp_i \Rvv+(z_{3-i})$.
		This proves the required isomorphism.
		In particular, $R^{\ver}(\p)$ is a formal power series
		ring over $\calO$, and
		$e(R^{\ver}(\p)/\vp)=1$.
	\end{proof}
	\begin{corollary}
		\label{cor-component-injection-2}
		There is an injection $$\Sigma_{0}(\bar\rho_0)\sqcup\Sigma_{k_2-k_1}(\bar\rho_0)\hookrightarrow \Sigma_{0}(\bar\rho_1)\sqcup\Sigma_{k_2-k_1}(\bar\rho_1)\sqcup\Sigma_{0}(\bar\rho_2)\sqcup\Sigma_{k_2-k_1}(\bar\rho_2).$$
	\end{corollary}
	\begin{proof}
		This follows directly from Proposition~\ref{prop-key-reducible}, and the map is given by $\p\mapsto\tp_i$, where $i\in\{1,2\}$ is determined by $\p$ as in Lemma~\ref{lem-relation-slope-and-reducible-type}(2). This map is injective because $\p=\tp_iR^{\ver}(\w,\t)+(z_{3-i})$. Note that the image of $\Sigma_{\goths}(\bar\rho_0)$ lies in $\Sigma_{\goths}(\bar\rho_1)\sqcup\Sigma_{\goths}(\bar\rho_2)$ for $\goths\in\{0,k_2-k_1\}$ and at most one of $\Sigma_{\goths}(\bar\rho_1)$ and $\Sigma_{\goths}(\bar\rho_2)$ is non-empty by our genericity assumption.
	\end{proof}
	\subsection{Proof of Theorem~\ref{thm-main-thm-split}}
	\label{finish-the-proof}
	We finish the proof of Theorem~\ref{thm-main-thm-split} in this subsection, and the main idea is similar to the proof of Theorem~\ref{thm-BM=1-non-split}, i.e. by counting argument. We first introduce some notation. For $i\in\{0,1,2\}$ and $*\in\{0,k_2-k_1\}$, put $n_{*}(\bar\rho_i):=\#(\Sigma_{*}(\bar\rho_i))$. We also put $n_{\irr}(\bar\rho_i):=\#(\Sigma_{(0,k_2-k_1)}(\bar\rho_i))$ and write $n(\bar\rho_i)$ for the number of irreducible components of the Kisin ring $A$ corresponding to $\bar\rho_i$, namely, $A=R^{\ver}(\w,\t)$ if $i=0$ and $A=R_i(\w,\t)$ if $i=1,2$. We write $e_i=e(A/\vp)$ and $e'_i=e'(\Spf(A)^{\rig})$ for its Breuil--M\'ezard multiplicity and number of distinct crystabelline slopes, respectively.
	
	Combining the calculations in Proposition~\ref{prop-BM-mult=halo-mult}, Lemma~\ref{lem-multiplicity-difference}, and Corollary~\ref{cor-split-compare-mult}, we list the data for all the possible cases in Table~\ref{tab:kisin-invariants}, from which we observe the following properties.
	\begin{enumerate}
		\item $n_{\irr}(\bar\rho_1)=n_{\irr}(\bar\rho_2)$.
		\item $e'_0=n_{\irr}(\bar\rho_i)+n_{0}(\bar\rho_1)+n_{k_2-k_1}(\bar\rho_1)+n_{0}(\bar\rho_2)+n_{k_2-k_1}(\bar\rho_2)$.
		\item $e_0=n_{\irr}(\bar\rho_i)+e_0'$.
	\end{enumerate}
	Recall that $e'_0$ provides a naive lower bound for $n(\bar\rho_0)$.
	Combining the decomposition (Definition~\ref{def-decomp-of-Min(R)}) of irreducible components in terms of slopes with Corollaries~\ref{cor-component-injection-1} and \ref{cor-component-injection-2}, we have 
	\begin{equation}
		\label{eq-6}
		\begin{aligned}
			e'_0\leq n(\bar\rho_0)&=n_{\irr}(\bar\rho_0)+n_{0}(\bar\rho_0)+n_{k_2-k_1}(\bar\rho_0)
			\\&\leq n_{\irr}(\bar\rho_i)+n_{0}(\bar\rho_1)+n_{k_2-k_1}(\bar\rho_1)+n_{0}(\bar\rho_2)+n_{k_2-k_1}(\bar\rho_2)
			\overset{(2)}{=}e_0'.
		\end{aligned}
	\end{equation}
	Hence the injections of Corollaries~\ref{cor-component-injection-1} and \ref{cor-component-injection-2} are \emph{bijections}. In particular, we have 
	\begin{equation}
		\label{eq-7}
		n_{\irr}(\bar\rho_0)=n_{\irr}(\bar\rho_i)\quad\textrm{for} \quad i=1,2, \quad\textrm{and}\quad n(\bar\rho_0)=e'_0.
	\end{equation}  So we obtain the first assertion of Theorem~\ref{thm-main-thm-split}. In particular, for any $\goths$ occurring in the generic fiber $\Spf(R^{\ver}(\w,\t))^{\rig}$, there exists exactly one minimal prime $\p\in\Sigma_{\goths}(\bar\rho_0)$.
	
	We then calculate the Breuil--M\'ezard multiplicity of each irreducible component of $A:=R^{\ver}(\w,\t)$. Note that Theorem~\ref{thm-main-thm-split}(2) follows from Proposition~\ref{prop-key-reducible} directly. We need to prove Theorem~\ref{thm-main-thm-split}(1). Combining Propositions~\ref{prop-key-irreducible} and \ref{prop-key-reducible} with Lemmas~\ref{lem multiplicity-components} and \ref{lem-1}, we have
	\begin{equation}
		\begin{aligned}
			e_0=&\sum_{\p\in\Min(A)}e(A/(\p,\vp))=n_{0}(\bar\rho_0)+n_{k_2-k_1}(\bar\rho_0)+\sum_{\p\in\Sigma_{(0,k_2-k_1)}(\bar\rho_0)}e(A/(\p,\vp))
			\\&\leq n_{0}(\bar\rho_0)+n_{k_2-k_1}(\bar\rho_0)+2n_{\irr}(\bar\rho_i)\overset{\eqref{eq-7}}{=}n(\bar\rho_0)+n_{\irr}(\bar\rho_i)\overset{\eqref{eq-6}+(3)}{=}e_0.
		\end{aligned}
	\end{equation}
	Consequently, all the inequalities about Breuil--M\'ezard multiplicities in Proposition~\ref{prop-key-irreducible} are indeed equalities, and we finish the proof of Theorem~\ref{thm-main-thm-split}.
	\begin{table}[H]
		\centering
		\renewcommand{\arraystretch}{1.6}
		\setlength{\tabcolsep}{4pt}
		\resizebox{\textwidth}{!}{
			\begin{tabular}{|l|c|c|c|c|c|}
				\hline
				Conditions
				& $(e_0,e'_0)$
				& $(e_1,e_2)=(e'_1,e'_2)$
				& $(n_0(\bar\rho_1),n_0(\bar\rho_2))$
				& $(n_{k_2-k_1}(\bar\rho_1),n_{k_2-k_1}(\bar\rho_2))$
				& $n_{\irr}(\bar\rho_1)=n_{\irr}(\bar\rho_2)$ \\
				\hline
				\shortstack[l]{
					$s_\e=0$,\\
					$k_0=\{s_\e+a\}$
				}
				& $(2e_1-2,e_1)$
				& $(e_1,e_1-2)$
				& $(1,0)$
				& $(1,0)$
				& $e_1-2$ \\
				\hline
				\shortstack[l]{
					$s_\e=0$,\\
					$k_0=\{s'_\e+a'\}$
				}
				& $(2e_1,e_1+1)$
				& $(e_1,e_1)$
				& $(1,0)$
				& $(0,1)$
				& $e_1-1$ \\
				\hline
				\shortstack[l]{
					$s_\e=0$,\\
					$k_0\notin\bigl\{\{s_\e+a\},\{s'_\e+a'\}\bigr\}$
				}
				& $(2e_1-1,e_1)$
				& $(e_1,e_1-1)$
				& $(1,0)$
				& $(0,0)$
				& $e_1-1$ \\
				\hline
				\shortstack[l]{
					$s'_\e=0$,\\
					$k_0=\{s_\e+a\}$
				}
				& $(2e_1,e_1+1)$
				& $(e_1,e_1)$
				& $(0,1)$
				& $(1,0)$
				& $e_1-1$ \\
				\hline
				\shortstack[l]{
					$s'_\e=0$,\\
					$k_0=\{s'_\e+a'\}$
				}
				& $(2e_1+2,e_1+2)$
				& $(e_1,e_1+2)$
				& $(0,1)$
				& $(0,1)$
				& $e_1$ \\
				\hline
				\shortstack[l]{
					$s'_\e=0$,\\
					$k_0\notin\bigl\{\{s_\e+a\},\{s'_\e+a'\}\bigr\}$
				}
				& $(2e_1+1,e_1+1)$
				& $(e_1,e_1+1)$
				& $(0,1)$
				& $(0,0)$
				& $e_1$ \\
				\hline
				\shortstack[l]{
					$s_\e\neq0,\ s'_\e\neq0$,\\
					$k_0=\{s_\e+a\}$
				}
				& $(2e_1-1,e_1)$
				& $(e_1,e_1-1)$
				& $(0,0)$
				& $(1,0)$
				& $e_1-1$ \\
				\hline
				\shortstack[l]{
					$s_\e\neq0,\ s'_\e\neq0$,\\
					$k_0=\{s'_\e+a'\}$
				}
				& $(2e_1+1,e_1+1)$
				& $(e_1,e_1+1)$
				& $(0,0)$
				& $(0,1)$
				& $e_1$ \\
				\hline
				\shortstack[l]{
					$s_\e\neq0,\ s'_\e\neq0$,\\
					$k_0\notin\bigl\{\{s_\e+a\},\{s'_\e+a'\}\bigr\}$
				}
				& $(2e_1,e_1)$
				& $(e_1,e_1)$
				& $(0,0)$
				& $(0,0)$
				& $e_1$ \\
				\hline
			\end{tabular}
		}
		\caption{Multiplicities, numbers of distinct slopes and of irreducible components.}
		\label{tab:kisin-invariants}
	\end{table}
	\begin{remark}
		\label{rmk-5}
		By \cite[Theorem~7.3]{paskunas-2adic} (together with the geometric Breuil--M\'ezard conjecture for non-split $\bar\rho_1$ and $\bar\rho_2$), the \emph{special fiber} of
		$R_i(\w,\t)$ is contained in the reducible locus. This gives
		a direct proof that
		$
		e(R^{\ver}(\p)/\vp)=2$ for $\p\in\Sigma_{(0,k_2-k_1)}(\bar\rho_0)$.
		Indeed, let $\tp_i$ be a corresponding minimal prime in Proposition~\ref{prop-key-irreducible}. The image of $c$ in
		$R^{\ps}/(I_{\irr}^{\ps},\vp)$ is nilpotent. Since
		$I_{\irr}^{\ps}\subseteq\tp_i$, its image in
		$
		R^{\ps}/(\tp_i,\vp)
		\simeq R_i(\w,\t)/(\tp_i,\vp)$
		is also nilpotent. As $R_i(\w,\t)/\tp_i$ is a formal power
		series ring over $\calO$, the image of $c$ in this special fiber is zero.
		Proposition~\ref{prop-key-irreducible} therefore gives
		\[
		R^{\ver}(\p)/\vp
		\simeq
		\FF[\![T_1,\dots,T_n,x,y]\!]/(xy),
		\qquad
		n=\dim(R^{\ps}/\tp_i)-1.
		\]
		Consequently, this \emph{special fiber} has multiplicity $2$ and
		exactly two irreducible components.
	\end{remark}
	
	\begin{remark}
		\label{rmk-6}
		Let $\p\in\Min(R^{\ver}(\w,\t))$, and let $\tp_i$ denote
		the corresponding minimal prime of $R_i(\w,\t)$ for each relevant $i$,
		as provided by Propositions~\ref{prop-key-irreducible} and
		\ref{prop-key-reducible}. Then
		$
		e(R^{\ver}(\p)/\vp)
		=
		\sum_i e\bigl(R_i(\w,\t)/(\tp_i,\vp)\bigr)$,
		where the sum runs over both $i=1,2$ if
		$\p\in\Sigma_{(0,k_2-k_1)}(\bar\rho_0)$, and over the single
		index $i\in\{1,2\}$ specified in \eqref{eq-4} otherwise.
		We may upgrade this equality to an equality of cycles via results in \cite[\S~7]{paskunas-2adic}.
		More precisely, put $d=\dim R^{\ver}(\w,\t)$, and recall 
		the cycle map defined in
		\cite[\S~7, Equation~(15)]{paskunas-2adic}
		$$\alpha:\mathscr{Z}_{d-2}(R_1/(c))\oplus\mathscr{Z}_{d-2}(R_2/(c))\hookrightarrow\mathscr{Z}_{d-1}(R^{\ver}/(c)),$$
		where $\mathscr{Z}_{\ell}(A)$ is the group of $\ell$-dimensional cycles of a Noetherian ring $A$. Write $\alpha_1(Z):=\alpha(Z,0)$ and $\alpha_2(Z):=\alpha(0,Z)$ for the restrictions to the two summands.	
		Applying the
		argument of \cite[Theorem~7.1]{paskunas-2adic} to each irreducible component gives
		\[
		z_{d-1}\bigl(R^{\ver}(\p)/\vp\bigr)
		=
		\sum_i
		\alpha_i\!\left(
		z_{d-2}\bigl(R_i(\w,\t)/(\tp_i,\vp)\bigr)
		\right),
		\]
		where $z_\ell(R)$ denotes the $\ell$-dimensional cycle associated to $R$.
	\end{remark}
	\appendix
	\section{Commutative algebra}
	
	In this section, we prove several lemmas in commutative algebra used
	in the main part of this paper.
	Let $A$ be a complete Noetherian local $\calO$-algebra with maximal
	ideal $\gothm$ and residue field $\FF$.
	We write $e(A)$ for the \emph{Hilbert--Samuel multiplicity} of $A$
	with respect to $\gothm$.
	
	\begin{lemma}
		\label{lem-ca1}
		Let $a\in\gothm$ be a non-zero-divisor, and set
		$
		\ord_{\gothm}(a):=\max\{j\geq 0:a\in\gothm^j\}$.
		Then we have
		$
		e(A/a)\geq \ord_{\gothm}(a)\cdot e(A)$.
		
	\end{lemma}
	
	\begin{proof}
		Since $a\in\gothm$ is a non-zero-divisor,
		$\dim(A/a)=\dim A-1$, so $a$ is part of a system of parameters
		of $A$; see, for example,
		\cite[Lemma~10.60.13]{stack} and its proof.
		The assertion then follows from \cite[Theorem~14.9]{mat}.
	\end{proof}
	
	\begin{lemma}
		\label{lem:mult leq 2}
		Let
		$
		A:=
		\frac{\FF[\![T_1,\dots,T_n]\!][\![x,y]\!]}{(xy-c)}$,
		where $c\in(T_1,\dots,T_n)$.
		Write
		$
		\ord(c):=
		\max\{j\geq 0:c\in(T_1,\dots,T_n)^j\}
		$.
		Then we have
		$
		e(A)=\min\{2,\ord(c)\}\leq 2$.
	\end{lemma}
	
	\begin{proof}
		Set $S=\FF[\![x,y,T_1,\dots,T_n]\!]$,
		$f=xy-c$, and $\gothm=(x,y,T_1,\dots,T_n)\subseteq S$.
		Since $S$ is a regular local ring and $f\neq 0$,
		\cite[(40.2)]{Nagata} gives
		$
		e(A)=\ord_{\gothm}(f)$.
		As $c$ involves only $T_1,\dots,T_n$, we have
		$
		\ord_{\gothm}(f)=\min\{2,\ord(c)\}$,
		which proves the assertion.
	\end{proof}
	
	Recall that $A$ is called \emph{unmixed} if
	$\dim A/\p=\dim A$ for every $\p\in\operatorname{Ass}(A)$.
	The next two lemmas are based on results of \cite{BM02}.
	
	\begin{lemma}
		\label{lem comm alg2}
		Let $A$ be a complete Noetherian local $\calO$-flat algebra
		that is equidimensional of dimension $n$ and has residue
		field $\FF$.
		Suppose that $A[1/\varpi]$ is reduced and $e(A/\varpi)=1$.
		Then
		$
		A\simeq\calO[\![X_1,\dots,X_{n-1}]\!]$.
	\end{lemma}
	
	\begin{proof}
		This is a slight generalization of \cite[Lemme~5.1.8]{BM02},
		which treats the case $n=2$ under the stronger assumption
		that $A[1/\varpi]$ is regular.
		By Nagata's multiplicity-one criterion
		\cite[Theorem~40.6]{Nagata}, a complete unmixed Noetherian
		local ring $R$ with $e(R)=1$ is regular.
		We first apply this criterion to show that $A$ is regular.
		
		Since $A$ is $\calO$-flat, $\varpi$ is a non-zero-divisor
		in $A$.
		Thus Lemma~\ref{lem-ca1} gives
		$
		\ord_{\gothm}(\varpi)\cdot e(A)
		\leq e(A/\varpi)=1$.
		It follows that $e(A)=1$ and
		$\varpi\in\gothm\setminus\gothm^2$.
		Moreover, $\calO$-flatness implies that the natural map
		$A\longrightarrow A[1/\varpi]$ is injective.
		Since $A[1/\varpi]$ is reduced, so is $A$.
		Consequently,
		$
		\operatorname{Ass}(A)=\operatorname{Min}(A)$.
		Together with the equidimensionality of $A$, this gives
		$\dim A/\p=\dim A$ for every $\p\in\operatorname{Ass}(A)$.
		Thus $A$ is unmixed, and Nagata's criterion implies that
		$A$ is regular.
		
		Since $A$ is regular of dimension $n$ and
		$\varpi\in\gothm\setminus\gothm^2$, we may choose
		$x_1,\dots,x_{n-1}\in\gothm$ such that
		$\bar\varpi,\bar x_1,\dots,\bar x_{n-1}$ form an
		$\FF$-basis of $\gothm/\gothm^2$.
		Nakayama's lemma then gives
		$
		\gothm=(\varpi,x_1,\dots,x_{n-1})$.
		By \cite[Theorem~7.16(a)]{eis}, there is a unique continuous
		homomorphism of $\calO$-algebras
		\[
		\phi\colon B:=\calO[\![X_1,\dots,X_{n-1}]\!]
		\longrightarrow A,
		\qquad X_i\longmapsto x_i.
		\]
		Since $\calO/(\varpi)\simeq\FF\simeq A/\gothm$ and
		$\gothm=(\varpi,x_1,\dots,x_{n-1})$,
		\cite[Theorem~7.16(b)]{eis} implies that $\phi$ is surjective.
		Since $A$ and $B$ are regular and $\phi$ identifies their
		chosen regular systems of parameters, the induced map on
		associated graded rings is an isomorphism.
		Hence \cite[Theorem~7.16(c)]{eis} implies that $\phi$ is
		injective, and we obtain the required isomorphism
		$
		\calO[\![X_1,\dots,X_{n-1}]\!]\simeq A$.
	\end{proof}
	
	\begin{lemma}
		\label{lem multiplicity-components}
		Let $A$ be a complete Noetherian local $\calO$-flat algebra
		that is equidimensional.
		Suppose that $A[1/\varpi]$ is regular.
		Write
		$
		\operatorname{Min}(A)
		=\{\mathfrak p_1,\dots,\mathfrak p_r\}$, and put
		$A_i:=A/\mathfrak p_i$.
		Then we have
		\[
		e(A/\varpi)
		=\sum_{i=1}^r e(A_i/\varpi)\geq r.
		\]
	\end{lemma}
	
	\begin{proof}
		By \cite[Lemma~5.1.6]{BM02}, we only need to check that the condition ``$A\sim \Pi_iA_i$" defined in \cite[Definition~5.1.3]{BM02} holds, namely
		\begin{enumerate}
			\item The natural map $A\rightarrow \Pi_iA_i\twoheadrightarrow A_i$ is surjective for any $i$, and 
			\item 
			the above map induces an isomorphism on the generic fiber $A[1/\varpi]\simeq\Pi_iA_i[1/\varpi]$.
		\end{enumerate}
		(1) is clear in our case and (2) follows from the reducedness of $A$ and the regularity of $A[1/\varpi]$. Indeed, it suffices to show the natural homomorphism $$\varphi:A[1/\varpi]\rightarrow \Pi_iA[1/\varpi]/\gothp_iA[1/\varpi]=\Pi_i A_i[1/\varpi]$$ is an isomorphism. The kernel of $\varphi$ is $\cap_i(\gothp_iA[1/\varpi])=(\cap_i\gothp_i)A[1/\varpi]=0$ as $A$ is reduced. The regularity of $A[1/\varpi]$ ensures that irreducible components of $\Spec A[1/\varpi]$ are pairwise disjoint, namely, $\gothp_iA[1/\varpi]+\gothp_jA[1/\varpi]=A[1/\varpi]$ for any $i\neq j$. By the Chinese remainder theorem we deduce that $\varphi$ is an isomorphism. This lemma is proved.
	\end{proof}
	
	\begin{lemma}
		\label{lem-minimal-primes}
		Let $A$ be a ring and let $I,J\subseteq A$ be ideals. 
		\begin{enumerate}
			\item If
			$\mathfrak p$ is minimal over $I\cap J$ and $\mathfrak p\in V(I)$, then $\mathfrak p$ is minimal over $I$.
			\item If
			$\mathfrak p$ is minimal over $I$ and $\mathfrak p\in V(I)-V(J)$, then $\mathfrak p$ is minimal over $I\cap J$.
		\end{enumerate}
	\end{lemma}
	\begin{proof}
		For (1),
		let $\mathfrak q$ be a prime ideal satisfying
		$
		I\subseteq\mathfrak q\subseteq\mathfrak p$.
		Then $I\cap J\subseteq I\subseteq\mathfrak q\subseteq\gothp$. Since $\mathfrak p$ is minimal
		over $I\cap J$, it follows that $\mathfrak q=\mathfrak p$. Hence
		$\mathfrak p$ is minimal over $I$.
		
		For (2), let $\mathfrak q$ be a prime ideal such that
		$
		I\cap J\subseteq \mathfrak q\subseteq \mathfrak p$.
		Since
		$
		IJ\subseteq I\cap J\subseteq \mathfrak q$
		and $\mathfrak q$ is prime, either $I\subseteq\mathfrak q$ or
		$J\subseteq\mathfrak q$. Our assumption $\gothp\notin V(J)$ implies that $J$ is not contained in $\mathfrak q$, hence $I\subseteq\mathfrak q$. Since
		$\mathfrak p$ is minimal over $I$, we conclude that
		$\mathfrak q=\mathfrak p$. Hence $\mathfrak p$ is minimal over
		$I\cap J$.
	\end{proof}
	
\end{document}